\documentclass[12pt,draftcls,onecolumn,letterpaper]{IEEEtran}

\IEEEoverridecommandlockouts                              

\usepackage{verbatim,color}

\usepackage{epsfig,verbatim,cancel}
\usepackage{amsmath}
\usepackage{epsfig,amsmath,amssymb}
\usepackage{longtable}
\usepackage{color}
\usepackage{mathrsfs}

\usepackage{algorithm,epsfig,amssymb,latexsym,color,amsmath,pifont,colordvi,multicol,verbatim}
\usepackage{amsthm}
\usepackage{multirow}
\newtheorem{proposition}{Proposition}
\newtheorem{theorem}[proposition]{Theorem}

\newtheorem{corollary}{Corollary}

\newtheorem{assumption}{Assumption}
\newtheorem{lemma}{Lemma}

\newtheorem{remark}{Remark}

\usepackage{amsmath, amssymb, bbm, xspace}

\newcommand{\ie}{i.e.\@\xspace} 

\newcommand{\Real}{\ensuremath{\mathbb{R}}}

\DeclareMathOperator*{\st}{subject\;to}

\def\half  {{\textstyle{1\over 2}}}

\def\spose#1{\hbox to 0pt{#1\hss}}

\def\text #1{\hbox{\quad#1\quad}}

\def\nthinsp{\mskip -2   mu}

\def\superstar{^{\raise 0.5pt\hbox{$\nthinsp *$}}}
\def\SUPERSTAR{^{\raise 0.5pt\hbox{$*$}}}

\def\lamstarT {\lambda^{\raise 0.5pt\hbox{$\nthinsp *$}T}}

\def\Fscr{{\cal F}}

\def\Lscr{{\cal L}}

\def\hbar{\skew{4.2}\bar h}

		\def\bkE{{\rm I\kern-.17em E}}
		\def\bkE{\mathbb{E}}
		\def\bk1{{\rm 1\kern-.17em l}}
		\def\bkD{{\rm I\kern-.17em D}}
		\def\bkR{{\rm I\kern-.17em R}}
		\def\bkP{{\rm I\kern-.17em P}}
		\def\bkY{{\bf \kern-.17em Y}}
		\def\bkZ{{\bf \kern-.17em Z}}

		\def\beq{\begin{eqnarray}}
		\def\bc{\begin{center}}
		\def\be{\begin{enumerate}}
		\def\bi{\begin{itemize}}
		\def\bs{\begin{small}}
		\def\bS{\begin{slide}}
		\def\ec{\end{center}}
		\def\ee{\end{enumerate}}
		\def\ei{\end{itemize}}
		\def\es{\end{small}}
		\def\eS{\end{slide}}
		\def\eeq{\end{eqnarray}}
		\def\qed{\quad \vrule height7.5pt width4.17pt depth0pt}

	\def\cp2problem#1#2#3#4{\fbox
		 {\begin{tabular*}{0.9\textwidth}
			{@{}l@{\extracolsep{\fill}}l@{\extracolsep{6pt}}l@{\extracolsep{\fill}}c@{}}
				#1 & & $#4 $ 
			\end{tabular*}}}

\newcommand{\pmat}[1]{\begin{pmatrix} #1 \end{pmatrix}}
		
		\renewcommand{\emph}[1]{\textbf{#1}}

		\def\bk1{{\rm 1\kern-.17em l}}
		\def\bkD{{\rm I\kern-.17em D}}
		\def\bkR{{\rm I\kern-.17em R}}
		\def\bkP{{\rm I\kern-.17em P}}
		
		\def\bkZ{{\bf{Z}}}

\newcommand {\beeq}[1]{\begin{equation}\label{#1}}
\newcommand {\eeeq}{\end{equation}}
\newcommand {\bea}{\begin{eqnarray}}
\newcommand {\eea}{\end{eqnarray}}

\def\texitem#1{\par\smallskip\noindent\hangindent 25pt
               \hbox to 25pt {\hss #1 ~}\ignorespaces}

\def\st{\mbox{subject to}}

\newcommand{\f}{\tilde{f}}
\newcommand{\g}{\tilde{g}}

\usepackage{verbatim,color}

\newcommand{\blue}[1]{{\bf \color{blue}#1}}
\newcommand{\vvs}[1]{{\color{black}#1}}
\newcommand{\vs}[1]{{\color{black}#1}}
\newcommand{\sv}[1]{{\color{black}#1}}
\newcommand{\us}[1]{{\color{black}#1}}
\newcommand{\uvs}[1]{{\color{black}#1}}
\newcommand{\uss}[1]{{\color{black}#1}}
\newcommand{\usss}[1]{{\color{black}#1}}
\newcommand{\yx}[1]{{\color{black}#1}}
\newcommand{\xy}[1]{{\color{black}#1}}

\newcommand{\yxb}[1]{{\color{black}#1}}
\newcommand{\yxc}[1]{{\color{black}#1}}

\newcommand{\yxe}[1]{{\color{black}#1}}
\newcommand{\yxf}[1]{{\color{black}#1}}

\newcommand{\yux}[1]{{\color{black}#1}}
\def\bE{\mathbb{E}}

\def\bP{\mathbb{P}}
\def\bE{\mathbb{E}}
\def\scrF{\mathcal{F}}

\def\bP{\mathbb{P}}
\def\bR{\mathbb{R}}

\allowdisplaybreaks

\begin{document}
\title{{\bf SI-ADMM}: A Stochastic Inexact ADMM Framework for Stochastic Convex Programs}
\author{Yue Xie and Uday V. Shanbhag\thanks{Xie and Shanbhag are with the Indust. and Manuf. Engg.,
					   Pennsylvania State University, University Park,
					   PA 16802, USA, {\tt\small yux111@psu.edu,
						   udaybag@psu.edu}. This work has been partially supported by NSF awards CMMI-1538605 and CMMI-1246887 (CAREER). Preliminary research by the authors has appeared in the Winter Simulation Conference (WSC)~\cite{DBLP:conf/wsc/XieS16}. Supplementary document is uploaded on arXiv(1711.05286) \cite{2017arXiv171105286X}. 
						   }
}
\maketitle
\begin{abstract}
We consider the structured stochastic convex program \vs{requiring the
minimization of} $\mathbb{E}[\tilde f(x,\xi)]+\mathbb{E}[\tilde g(y,\xi)]$
\vs{subject to the constraint} $Ax + By = b$. Motivated by the need for \vs{decentralized
schemes and structure,} 
 we propose a stochastic inexact ADMM ({\bf SI-ADMM}) \vvs{framework} where
subproblems are solved inexactly via stochastic approximation schemes. Based on
this framework, we prove the following: (i) under suitable assumptions on the
\vs{associated batch-size of samples utilized at} each iteration, the SI-ADMM \vs{scheme} produces a
sequence that converges to the unique solution almost surely; (ii) If the
number of gradient steps (\sv{or equivalently, the number of sampled gradients}) \vvs{utilized for solving the subproblems}
{in each iteration} increases at a geometric rate, the mean-squared error
diminishes to zero at a prescribed geometric rate; (iii) \vvs{The overall 
iteration} complexity \vvs{in terms of  gradient steps (or equivalently samples)}
\vs{is found to be consistent} with the \vvs{canonical level of
$\mathcal{O}(1/\epsilon)$.} \vvs{Preliminary applications on {\sc
LASSO} and distributed regression suggest that the scheme performs well
compared to its competitors.}
\end{abstract}

\section{Introduction}
{Over the last decade, structured optimization problems have been
addressed via a \vvs{subclass} of Lagrangian schemes, namely the alternating
direction method of multipliers
(ADMM)~\cite{boyd2011distributed}. 
Such schemes have grown {immensely} popular, particularly in resolving a
host of \uss{structured} machine learning and image processing problems  
such as image recovery~\cite{afonso2010fast},
robust PCA~\cite{DBLP:journals/corr/LinCM10}, low-rank
representation~\cite{lin2011linearized} (cf.~\cite{boyd2011distributed}
for a comprehensive review.)  {Typically, ADMM is applied towards structured {deterministic}
	convex optimization problems of the form:
\begin{align}\label{DOpt}
\min_{x \in \mathcal{X},y \in \mathcal{Y}} \ f(x) + g(y) \quad \st \  Ax + By = b.
\end{align}
{We} consider {a stochastic generalization leading to a} structured stochastic convex program:}
\begin{align} \tag{SOpt}\label{SOpt}
\min_{x \in \mathcal{X}, y \in \mathcal{Y}} \ \mathbb{E}[\tilde f(x,\xi)] + \mathbb{E}[\tilde g(y,\xi)] \quad \st \
 Ax + By = b,
\end{align}
where $\xi:\Omega \to \Real^d$, $\tilde f: \mathcal{X} \times \Real^d \to \Real$, 
$\tilde g: \mathcal{Y} \times \Real^d \to \Real$, $A \in \Real^{p \times n}$, $B \in
\Real^{p \times m}$, $b \in \Real^p$, and
$(\Omega,\Fscr,\mathbb{P})$ denotes the probability space.  Furthermore,
	we assume that $\tilde f(\cdot,\xi)$ and
	$\tilde g(\cdot,\xi)$ are convex in $(\cdot)$ for
	every $\xi \in \Xi \subseteq \bR^d$, where $\vs{ \Xi \triangleq \{\xi(\omega): \omega \in \Omega\}}$.  
We assume that $\mathcal{X} \triangleq \bR^n$ and $\mathcal{Y} \triangleq \bR^m$ throughout this
paper.	
\yx{Popular \vs{Monte-Carlo sampling} approaches} for the solution of \eqref{SOpt} 
\vs{include} sample-average
approximation~\cite{shapiro09lectures} {and} stochastic approximation
schemes~\cite{robbins51sa,spall,nemirovski_robust_2009,Farzad1}.
We also note that the recent work by Pasupathy et al.~\cite{pasupathy2018sampling}
considers sequential sampling concerns, an issue that assumes relevance in this 
paper.

\vvs{\bf Motivation.} \vvs{There has been significant effort in extending ADMM to multi-block and nonconvex
regimes~\cite{wang2019global,doi:10.1137/140990309} but far
less exists on contending with expectation-valued objectives.} Motivated by
\vvs{this lacuna as well} as the following benefits: (1) ADMM schemes
display strong theoretical properties and computational performance in the
resolution of structured constrained optimization problems (see~Boyd et
al.~\cite{boyd2011distributed}); (2) this avenue allows for problem structure
to be exploited {via} distributed computation, this research considers
addressing \eqref{SOpt} by adapting an existing ADMM architecture to stochastic
setting. 

{Over the last several years, there has been a surge of interest in developing and customizing ADMM schemes to contend with a range of problems in decision and control theory, including distributed control~\cite{admm7,admm10}, model-predictive control~\cite{admm11}, system identification~\cite{admm12}, discrete-valued control~\cite{admm13}, control design~\cite{admm14, admm7}, and  consensus techniques~\cite{admm5,admm8,admm10}. We believe that the proposed models could be enriched and broadened to accommodate large datasets, \vs{achievable by} allowing for expectation-valued objectives \vs{as} managed by the schemes proposed in this paper.}

\uvs{\bf Stochastic generalizations.} {Prior work~\cite{DBLP:journals/corr/WangB13a,ouyang2013stochastic} has focused on the following problem:} 
\begin{align}\label{Regularized ERM}
\min_{x\in \mathcal{X},y \in \mathcal{Y}} \ \bE_{\xi} [\theta_1(x,\xi)]+\theta_2(y) \quad \st \ Ax+By = b.
\end{align} 
This problem has arisen from minimizing the regularized expected
risk metric~\cite{vapnik1998statistical,ouyang2013stochastic}, where
$\theta_1(x) \triangleq \mathbb{E}[\theta_1(x,\xi)]$ {denotes} the expected
loss function {and} $\theta_2(y)$ represents the regularizer. \yxe{Variants of stochastic ADMM~\cite{DBLP:journals/corr/WangB13a, ouyang2013stochastic, azadi2014towards, gao2018information, Chen2018} have been proposed for resolving~\eqref{SOpt}}. \vs{Amongst} these,
Ouyang et al.~\cite{ouyang2013stochastic} 
derive rates in terms of sub-optimality plus
infeasibility: ${\cal O}(1/\sqrt{T})$ when $\theta_1(x)$ is convex and smooth, and ${\cal O}(\log(T)/T)$ when $\theta_1(x)$ is strongly convex, where $T$ denotes the iteration index (cf. ~\cite{azadi2014towards,Chen2018} for slightly improved
rates)}. \yxe{Gao et al.~\cite{gao2018information} obtain similar results when both $\theta_1$ and $\theta_2$ assumed to be expectation-valued. Note that all of these schemes \vs{as well as SI-ADMM} can address general distributions} while 
\yxe{finite-sum problems are considered in~\cite{zhong2014fast, zheng2016fast, liu2017accelerated}}.

\vvs{\bf Contributions.} {We consider the resolution of \eqref{SOpt}
\vvs{via} a generalized ADMM scheme}~\cite{deng2016global}, where a linear rate of convergence was proven under
suitable strong convexity assumptions. 
We modify this scheme to \yxc{a stochastic inexact ADMM} {\vvs{\bf (SI-ADMM)} \vs{that \vvs{requires} computing inexact solutions} to the \yx{subproblems} at each
iteration, where each subproblem (a stochastic optimization
problem) is resolved inexactly by a stochastic approximation method, requiring a finite (but) increasing number of
gradient steps. Based on
such a framework, we make the following contributions:\\
{\bf (i)} \textbf{Unconstrained stochastic optimization}: {We derive an optimal rate for stochastic approximation on an unconstrained problem when the noise is state dependent}. \\
{\bf (ii)} {\bf \vs{Almost-sure} convergence:} 
Under suitable assumptions on the number of samples utilized at each iteration, the sequence of iterates \us{is shown to converge a.s.} to the unique solution of the problem. \\
{\bf (iii)} {\bf Geometric rate of convergence:} When the number of gradient steps (or equivalently sampled gradients) increases at a geometric rate, the mean-squared error of major iterates diminishes to zero at a {prescribed} geometric rate. Moreover, a canonical overall iteration (sample) complexity is demonstrated.

\vvs{Table~\ref{tab:SADMM results} \vvs{compares our scheme (SI-ADMM) with SADM0~\cite{ouyang2013stochastic} and SADM1~\cite{azadi2014towards} under different assumptions on $f$ and $g$ for problem \eqref{DOpt}. Note that SI-ADMM can accommodate expectation-valued $g$ and provides optimal rate and overall iteration complexity statements.}}

\begin{table}[htbp]
{\tiny
\begin{center}
\begin{tabular}{| c | c | c | c |}
	\hline
	{\bf Algorithm}               & {\bf SADM0 \cite{ouyang2013stochastic}}                    & {\bf SADM1 \cite{azadi2014towards} }                                  & {\bf SI-ADMM}                                     \\ \hline
	\multirow{2}{*}{$f$}    & Stochastic               & Stochastic                             & Stochastic                                  \\
	                        & SC + \yx{D}                  & SC + D                                 & SC + LG                                     \\ \hline
	\multirow{2}{*}{$g$}    & deterministic            & deterministic                          & Stochastic                                  \\
	                        & C + nonsmooth            & C + nonsmooth                          & SC + LG                                     \\ \hline
	Rate                    & $\mathcal{O}(\log(T)/T)$  & $\mathcal{O}(1/T)$                     & linear{$^\dagger$}                        \\ \hline
	\multirow{2}{*}{Metric} & suboptimality            & suboptimality                          & \multirow{2}{*}{$\bE[\| u_T - u^* \|^2_G]$} \\
	                        & + infeasibility          & + infeasibility                        &  \\ \hline
	a.s. conv.              & $\times$                 & $\times$                               & $\checkmark$                                \\ \hline
	complexity                & \yxe{$\mathcal{O}( \frac{1}{\epsilon} \log(\frac{1}{\epsilon} ))$ }            & \yxe{$\mathcal{O}(\frac{1}{\epsilon})$} & $\mathcal{O}(\frac{1}{\epsilon})$                   \\ \hline
	\multirow{2}{*}{Notes}                   & \multirow{2}{*}{$\mathcal{X}$ is compact} & $\mathcal{X}$ is compact & $\mathcal{X} = \bR^n$  \\
	& &$\mathcal{Y}$ is compact & $\mathcal{Y} = \bR^m$\\ \hline
\end{tabular}
\end{center}
\caption{Summary of results on stochastic variants of ADMM} 
\vspace{-0.2in}
\footnotesize (SC) C: (Strongly) Convex,  D: Differentiable, LG: \yux{Continuously differentiable with Lipschitz continuous gradient}; \textit{suboptimality + infeasibility:} $\bE[f(\bar x_T) + g(\bar y_T) - f(x^*) - g(y^*) + \bar \rho \| A \bar x_T + B \bar y_T - b \| ] $, where $\bar \rho$ \vs{denotes} the upper bound for dual variables in \cite{azadi2014towards}, $\bar x_T, \bar y_T$ \vs{denote ergodic averages}, 
and $(x^*,y^*)$ is the optimal solution. \yxe{Complexity of SADM0 and SADM1 estimated using this metric.} 
Complexity of SI-ADMM is overall iteration/sample complexity to obtain $\epsilon$-solution $\bar u$ such that $\bE[\| \bar u - u^* \|^2_G ] \leq \epsilon$.  \yxe{$\dagger$Linear rate of SI-ADMM with respect to major iterates.}
}
\label{tab:SADMM results}
\vspace{-0.2in}
\end{table}

Finally, we implement our scheme on LASSO and distributed regression~\cite{sundhar2012new}, \vs{where} 
\vs{we observe that SI-ADMM performs relatively well on the problems considered.} 

The remainder of the paper is organized into
four sections. In Section~\ref{sec: SI-ADMM}, we
review the linear {rate} \vvs{statement} {for deterministic ADMM}, {derive the optimal rate for SA on \vs{unconstrained problems}}, and {formally define the stochastic inexact ADMM scheme.}  
In section~\ref{sec: convergence and rate-SIADMM}, asymptotic convergence
and rate statements are {developed} \sv{while} {preliminary numerics} are
presented in section~\ref{sec: Numerics-SIADMM}. We
conclude with a short summary in
Section~\ref{sec: conclusion-SIADMM}.

 \textbf{Notation:} $\lambda_{\rm max}(M)$ and $\lambda_{\rm min} (M)$ denote
the largest and smallest eigenvalue of matrix $M$, respectively. Given $z \in
\bR^n, M \in \mathcal{S}^n$, $\|z\|_M^2 \triangleq z^TMz$. When $M \in
\mathcal{S}_+^n$, $\|z\|_M \triangleq \sqrt{z^TMz}$. ${\bf 0}$ denotes \sv{a} vector
or matrix containing \sv{zeros} with \sv{the appropriate} dimension, \sv{while} $\log$ denotes logarithm
\sv{to} base $e$. $(v_1; \hdots; v_k) \triangleq (v_1^T, \hdots, v_k^T)^T$. 


\section{A stochastic ADMM scheme}\label{sec: SI-ADMM}
In section~\ref{subsec: g-ADMM}, we {review a 
convergence statement for deterministic ADMM schemes}. {We then derive} a
{rate statement} for unconstrained stochastic optimization in
Section~\ref{subsec:uncon-SP}.
In Section~\ref{subsec: SI-ADMM},  we present 
a stochastic inexact ADMM (SI-ADMM) scheme that {requires resolving} unconstrained problems at each step.
{Finally, we derive some useful recursive inequalities
 in Section~\ref{subsec: choice of sample size}}.
		
\subsection{A deterministic generalized ADMM scheme}\label{subsec: g-ADMM}
A generalized ADMM scheme (Algorithm~\ref{generalized admm}) {that can resolve
	\eqref{SOpt}} was suggested in~\cite{deng2016global}. Suppose the
		augmented Lagrangian function  
$\mathcal L_\rho (x,y,\lambda)$ is defined as 
\begin{align}\label{Def: ALfunc}
\mathcal L_\rho (x,y,\lambda) \triangleq & f(x) + g(y) - \lambda^T(Ax+ By-b) 
 + \frac{\rho}{2} \|Ax + By-b\|^2,
\end{align}
where $f(x) \triangleq \bE[\f(x,\xi)]$ and $g(y) \triangleq \bE[\g(y,\xi)]$, both of which are 
convex due to the \sv{assumptions on} \eqref{SOpt}. \yxe{In Algorithm~\ref{generalized admm}, \vs{$\gamma > 0$ denotes the stepsize for the multiplier update.} 
}
\begin{algorithm}[htbp]
\caption{{\bf g-ADMM:} Generalized ADMM scheme}
\label{generalized admm}
\begin{enumerate}
\item[(0)] 
	\uvs{Choose matrices $P, Q$ and let $k=0$;
Given $x_0, y_0,\lambda_0, \rho > 0$, $\gamma > 0$};
\item[(1)] Let $x_{k+1}, y_{k+1}, \lambda_{k+1}$ be given by the following:
\begin{align}
\tag{$y$-Update} \label{yexactupdate}	y_{k+1} &:= \displaystyle
\mbox{arg}\hspace{-0.025in}\min_y  \left(\mathcal L_\rho(x_k,y,\lambda_k) + \half \| y-y_k \|_Q^2 \right) \\
\tag{$x$-Update} \label{xexactupdate}	x_{k+1} &: = \displaystyle \mbox{arg}\hspace{-0.025in}\min_x \left(\mathcal L_\rho(x,y_{k+1},\lambda_k) + \half \| x-x_k \|_P^2 \right) \\
\tag{$\lambda$-Update}	\lambda_{k+1} &:= \lambda_k - \gamma \rho( Ax_{k+1} + By_{k+1} - b). 
\end{align}
\item [(2)] \yxf{$k:=k+1$ and return to (1).} 
\end{enumerate}
\end{algorithm}
Moreover, we make the following assumptions, all of which are necessary for \sv{proving the}
\textbf{global linear} convergence of the sequence of iterates
produced by Algorithm \ref{generalized admm}. Of these, the first pertains to the
existence of a {KKT} point to the original optimization problem: 
\begin{assumption} \label{Ass: existence of KKT point-SI-ADMM} There exists a \xy{KKT} point $u^* \triangleq
(x^*;y^*;\lambda^*)$ to problem (SOpt); i.e., $x^*,y^*,\lambda^*$ satisfy the KKT conditions:
 $A^T \lambda^*   \in \partial f(x^*)$, 
 $B^T \lambda^*  \in \partial g(y^*)$, $Ax^* + By^* = b.$
\end{assumption}
The second {assumption imposes} convexity and Lipschitzian assumptions \vs{on}
$f$ and $g$ {where using the} notation $\nabla_x f$ or $\nabla_y g$
implies that {either} $f$ or $g$ is differentiable. Let us define $\hat{P}$ as 
\begin{align}\label{def: Phat}
\hat P \triangleq P + \rho A^TA.
\end{align}
\begin{assumption}\label{Ass: objective function-SI-ADMM} {$\vs{Q,\hat P \succeq 0}$ are
	symmetric matrices. Additionally,}  {one of the following holds.}\\
\textbf{(a)} $f(x)$ is strongly convex and $\nabla_x f(x)$ is
	Lipschitz continuous in $x$ on {$\Real^n$}. $A$ has full row rank. Additionally, $B$ has full column rank whenever $Q \succ {\bf 0}$.\\
\textbf{(b)} $f(x)$ and $g(y)$ are strongly convex in $x$ {on $\Real^n$ and in $y$ on $\Real^m$, respectively. Furthermore,  $\nabla_x
		f(x)$ is a Lipschitz continuous function in $x$ on $\Real^n$ and $A$ has full row rank.} \\
\textbf{(c)} {$f(x)$ is strongly convex and $\nabla_x f(x)$ is Lipschitz continuous in $x$ on $\Real^n$. Further, 
$\nabla_y g(y)$ is Lipschitz continuous in $y$ on $\Real^m$ and $B$
has full column rank.} \\
\textbf{(d)} {The functions $f(x)$ and $g(y)$ are strongly convex in $x$ on $\Real^n$ and in $y$ on $\Real^m$, respectively. Furthermore,
	$\nabla_x f(x)$ and $\nabla_y g(y)$ are Lipschitz continuous in $x$ on $\Real^n$ and in $y$ on $\Real^m$, respectively}.
\end{assumption}
The next assumption ensures that Algorithm~\ref{generalized admm} generates a bounded sequence. 
\begin{assumption}\label{Ass: P&Q-Si-ADMM}
Either one of the following holds:
(i) $\hat P \succ  {\bf 0}, Q \succ {\bf 0}$; or (ii) $\hat P \succ {\bf 0}, Q = {\bf 0}$, $B$ has full column rank.
\end{assumption}
Next, we review the main result from~\cite{deng2016global} but 
require some definitions. \sv{Let} $u_k \triangleq (x_k; y_k; \lambda_k)$ \sv{and} 
\begin{align}\label{Def: G} 
G \triangleq \pmat{P +\rho A^TA & & \\ & Q & \\ & & \frac{1}{\rho \gamma} I_p}.
\end{align}
\begin{theorem}[{\bf\cite[Theorem 3.1, 3.2]{deng2016global}}]  \label{linear convergence of g-ADMM}
\textit{Suppose Assumptions~\ref{Ass: existence of KKT point-SI-ADMM}, \ref{Ass: objective function-SI-ADMM}, and~\ref{Ass: P&Q-Si-ADMM} hold. \yux{In addition, suppose that
$\gamma$ and $P$ satisfy the following.
\begin{align} \label{gammaP}
\begin{aligned}
\mbox{(i)} &\  \mbox{If $P \neq {\bf 0}$ then $ (2-\gamma)P \succ (\gamma - 1)\rho A^TA$;} \\
\mbox{(ii)} &\ \mbox{If $P = {\bf 0}$ then $\gamma = 1$. }
\end{aligned}
\end{align}} 
Then there exists a KKT point of \eqref{SOpt} denoted by $u^* = (x^*; y^*; \lambda^*) $, such that $\|u_k-u^*\|_G \rightarrow 0$ as $k \to \infty$.
	Furthermore, there exists  $\delta > 0$ such that 
\begin{equation}\label{contraction of iterates}
 \|u_{k+1}-u^*\|_G^2 \leq \frac{1}{1+\delta} \|u_k - u^*\|_G^2.
\end{equation}
}
\end{theorem}
Note that {if $Q$ and $\hat{P}$ are chosen to be positive definite,
	then $\|\bullet\|_G$ reduces to a norm, rather than a  semi-norm}. {
		The choice of $\delta$ {is examined in} the next Corollary \vs{where $\mu_f$ and $L_f$ denote the convexity and Lipschitz constants of $f$ and $\nabla_x f$, respectively}.
\begin{corollary}[{\bf\cite[Cor~3.6]{deng2016global}}] \label{choice of delta - SI-ADMM}
Suppose Assumption~\ref{Ass: existence of KKT point-SI-ADMM}, \vs{\ref{Ass: objective function-SI-ADMM} (a)}, \ref{Ass: P&Q-Si-ADMM} hold, 
 $P = {\bf 0}$, \vs{and} $\gamma = 1$.\\
(i) {If} $Q = {\bf 0}$, the sequence $\{ u_k \}$ generated by Algorithm~\ref{generalized admm} satisfies \eqref{contraction of iterates} with
$
\delta \vs{ \ \triangleq \ } 2 \left( \frac{\rho \| A \|^2}{\mu_f} + \frac{L_f}{\rho \lambda_{\min}(AA^T)} \right)^{-1}.
$\\
(ii) If $g$ is strongly convex with constant $\sigma_g$ and $Q
			 \succ {\bf 0}$, then \eqref{contraction of iterates} is satisfied with $ \delta \triangleq \min \{ \delta_0, 2 \sigma_g / \|
			 Q \| \}$, where 
$\delta_0 \triangleq 2 \left( \frac{\rho \| A \|^2}{\mu_f} + \frac{L_f}{\rho
		\lambda_{\min}(AA^T)} \right)^{-1}$. 
\end{corollary}
\begin{proof}
(i) See \cite[Corollary 3.6]{deng2016global}.
(ii) From inequality (3.21) of \cite{deng2016global}, \vs{for all $t \in [0,1]$, we have}
\begin{align*}
 \| u_k - u^* \|_G^2 - \| u_{k+1} - u^* \|_G^2 
& \geq \ \overbrace{2t\mu_f \| x_{k+1} - x^* \|^2 + 2(1-t)\frac{\lambda_{\min}(AA^T)}{L_f}\| \lambda_{k+1} - \lambda^* \|^2}^{\rm Term (A)} \\
& + 2\sigma_g\| y_{k+1} - y^*\|^2 + \| u_k - u_{k+1} \|_G^2 + 2\mu_f\| x_k - x_{k+1}\|^2,
\end{align*}
and inequality (3.23) of \cite{deng2016global}:
$\mathrm{Term (A)} \geq \delta_0 \left( \rho \| A \|^2\|x_{k+1} - x^*\|^2 + \frac{1}{\rho} \| \lambda_{k+1} - \lambda^* \|^2 \right),$
\vs{implying that}
\begin{align*}
& \| u_k - u^* \|_G^2 - \| u_{k+1} - u^* \|_G^2 \\
& \geq \delta_0 \left( \rho \| A \|^2\|x_{k+1} - x^*\|^2 + \frac{1}{\rho} \| \lambda_{k+1} - \lambda^* \|^2 \right) + 2\sigma_g\| y_{k+1} - y^*\|^2 \\
& = \delta_0 \left( \rho \| A \|^2\|x_{k+1} - x^*\|^2 + \frac{1}{\rho} \| \lambda_{k+1} - \lambda^* \|^2 + \frac{2\sigma_g}{\delta_0 \|Q\|} \|Q\| \| y_{k+1} - y^*\|^2 \right) \\
& \geq  \delta_0 \min \left\{1, \frac{2\sigma_g}{\delta_0 \|Q\|} \right\} \left( \rho \| A \|^2\|x_{k+1} - x^*\|^2 + \frac{1}{\rho} \| \lambda_{k+1} - \lambda^* \|^2 + \|Q\| \| y_{k+1} - y^*\|^2 \right) \\
& \geq \min \left\{\delta_0, \frac{2\sigma_g}{\|Q\|} \right\} \| u_{k+1} - u^* \|_G^2 
 = \delta  \| u_{k+1} - u^* \|_G^2.
\end{align*}
\end{proof}
\begin{remark}
Cor.~\ref{choice of delta - SI-ADMM} is {a consequence of modifying} \cite[Cor.~3.6]{deng2016global} to include the case \vs{where } $Q \succ {\bf 0}$ \vs{and} $g$ \sv{is} strongly convex. 
{Additionally, according to \cite{deng2016global}, if the matrix $[A\ B]$ has full row rank, $\delta$ is determined by problem and algorithm parameters satisfying assumptions of Theorem~\ref{linear convergence of g-ADMM} and is 
independent of the choice of the initial point.} \qed
\end{remark}
{We define} {$y^{\rm exact}(\bullet), x^{\rm exact}(\bullet)$, and
$\lambda^{\rm exact}(\bullet)$ as the update maps} in each iteration such that
$ y_{k+1} = y^{\rm exact}(x_k, y_k, \lambda_k)$, $x_{k+1} =
x^{\rm exact}(x_k, y_{k+1}, \lambda_k)$, and  $\lambda_{k+1} = \lambda^{\rm exact}(x_{k+1}, y_{k+1}, \lambda_k).$ Furthermore, {the} map
$\Gamma^*$ is defined as follows: \begin{align}\label{def-Gamma-SI-ADMM}
	\Gamma^*(u) = \pmat{ y^{\rm exact}(x,y,\lambda) \\
						 x^{\rm exact}(x,\bar y^{\rm exact},\lambda) \\
						 \lambda^{\rm exact}(\bar x^{\rm exact}, \bar y^{\rm exact}, \lambda)},
\end{align}
where $u \triangleq (x;y;\lambda)$, $\bar y^{\rm exact} \triangleq y^{\rm exact}(x,y,\lambda)$ and $\bar x^{\rm exact} \triangleq
x^{\rm exact}(x,\bar y^{\rm exact}, \lambda).$ Furthermore, under Assumption~\ref{Ass: objective function-SI-ADMM} and \ref{Ass: P&Q-Si-ADMM}, 
	it can be seen that $y$ update and
$x$ update in Algorithm \ref{generalized admm} have unique solutions. {Thus,
	for each given input $u_k$, each iteration of Algorithm \ref{generalized admm}
	provides a unique output, denoted as $u_{k+1}^*$, {where}
	$ u_{k+1}^* := \Gamma^* (u_k). $ Thus, the map $\Gamma^*(\bullet)$
	is a well-defined single-valued map under the three assumptions. 
	Unfortunately, when the expectation
	$\mathbb{E}[\bullet]$ is over a general measure space, the ADMM
	scheme is not practically implementable since the $x$ and $y$
	updates require exact solutions of stochastic optimization problems.
This motivates an implementable stochastic generalization of this
scheme. {We conclude this subsection by restating the supermartingale convergence lemma {that}
allows for deriving a.s. convergence of the sequence.}
\begin{lemma}[{\bf\cite[Lemma 10, pg. 49]{polyak1987introduction}}] \label{lm: supermartingale}
\textit{Let $ \{ \nu_k \} $ be a sequence of nonnegative random variables, where $ \bE [ \nu_0 ] < \infty $. Let $ \{ \tau_k \} $ and $ \{ \mu_k \}$ be deterministic scalar sequences  such that \vs{$\tau_k \in (0,1)$ and $\mu_k \geq 0$ \vs{where} } 
$ \bE[ \nu_{k+1} \mid \nu_0 , \hdots, \nu_k] \leq (1 - \tau_k) \nu_k + \mu_k$ \vs{a.s.} for all $k \geq 0$, 
$\sum_{k = 0}^\infty \tau_k = \infty$, $\sum_{k=0}^\infty \mu_k < \infty$, and $\lim_{k \rightarrow \infty} \frac{\mu_k}{ \tau_k} =  0.$
Then $\nu_k \rightarrow 0$ almost surely as $k \rightarrow \infty$.}
\end{lemma}

\subsection{\vs{Stochastic approximation} for unconstrained stochastic
	convex programs}\label{subsec:uncon-SP}
Before proceeding to the stochastic {generalization}, we 
{derive a formal rate statement when employing stochastic
approximation on unconstrained stochastic convex programs} under
{relatively weaker} assumptions on the variance of gradients. Specifically, consider the \yxf{unconstrained} stochastic program:
\begin{align}\label{uncon-SP} \tag{uncon-SP}
\vs{\min_{x \in \mathbb{R}^n}}  \vs{ F(x), \mbox{ where } F(x) \triangleq \bE[\tilde F(x,\xi)]},
\end{align}
where $F: \bR^n \rightarrow \bR$, $\xi: \Omega \rightarrow \bR^d$,
	  $\tilde{F}: \bR^n \times \bR^d \rightarrow \bR$ and
	  $(\Omega, \scrF, \bP)$ denotes the probability space. {A
		   stochastic approximation (SA) \vs{scheme} for solving \eqref{uncon-SP} is defined next.}
\begin{align}\label{SA-SI-ADMM}\tag{SA}
x_{k+1} := x_k - \gamma_k \yux{\nabla_x} \tilde{F}(x_k,\xi_k), \quad {\forall k \geq 1},
\end{align}
where 
$\gamma_k > 0$ denotes the steplength. 
Suppose the history is captured by $\scrF_k \triangleq
	\{ \yxe{x_1}, \xi_1,\hdots,\xi_{k-1}\}$ for $k \geq 2$, $\scrF_1
		\triangleq \{x_1\}$. Furthermore, \yux{$\xi_1, \hdots, \xi_k$ is a sequence of random variables} and 
\yxf{
\begin{align} \label{Def: w_k}
w_k \triangleq \yux{\nabla_x} \tilde{F}(x_k,\xi_k) -
	\nabla F(x_k).
\end{align}	}
	We \vs{assume the following on} $F(x)$ and $\tilde
	F(x,\xi)$ and proceed to prove convergence \sv{of the sequence} to $x^*$, the {unique} optimal solution of \eqref{uncon-SP}.
\begin{assumption}\label{Ass: uncon-sp}
 
(1) $F(x)$ is $c-$strongly convex and 
 differentiable {with $L-$Lipschitz continuous gradients}. \vs{In addition, $\tilde F(\cdot,\xi)$ is convex for every $\xi \in \bR^d$}.\\
(2) \yux{$w_k$ satisfies the following two requirements:}\\
\yxf{(i)} (Unbiasedness). \yux{$\bE[ w_k \mid \scrF_k ] = \bold{0}$, a.s., $\forall k \ge 1$.}\\
\yxf{(ii)} (\sv{State-dependent noise}). 
 \yux{There exist scalars $v_1, v_2 > 0$ such that $\bE[ \| w_k \|^2 \mid \scrF_k] \leq v_1 \| x_k \|^2 + v_2$, a.s., $\forall k \ge 1$.}
\end{assumption} 
\begin{proposition}\label{uncons-recursive}
Suppose Assumption~\ref{Ass: uncon-sp} holds \vs{and} 
$\{ x_k \}$ denotes the sequence generated by \eqref{SA-SI-ADMM}. \vs{Then the following hold.}

\noindent (i) For any ${r} > 0$, \yux{the following} {holds \vs{a.s.} for all $k \geq 1$}:
\begin{align}\label{superMatingale-unconsp}
\notag \bE[\|x_{k+1} - x^* \|^2 \mid \scrF_k]  & \leq  \left(1 - 2c\gamma_k +
		\gamma_k^2 \left( L^2 + v_1\left(1+\frac{1}{\vvs{r}}\right) \right) \right) \|x_k
- x^* \|^2 \\ 
&  + \gamma_k^2(v_1(1+\vvs{r}) \| x^* \|^2 + v_2);
\end{align}
\noindent (ii) If $\{\gamma_k\}$ is a non-summable but square
	summable positive sequence, then $\{x_k\} \xrightarrow{k \to \infty} x^*$ a.s..
\end{proposition}
\begin{proof}
\noindent (i) {By definition of $x_{k+1}$ {from (SA)}, we have that}
\begin{align*}
 \ \| x_{k+1} - x^* \|^2 
& = \| x_k - \gamma_k \yux{\nabla_x} \tilde F (x_k,\xi_k) - x^* \|^2  \\
& = \| x_k - x^* \|^2 - 2 \gamma_k \yux{\nabla_x} \tilde F (x_k,\xi_k)^T (x_k - x^*) + \gamma_k^2 \| \yux{\nabla_x} \tilde F (x_k,\xi_k) \|^2 \\
& = \| x_k - x^* \|^2 - 2 \gamma_k \nabla F (x_k)^T (x_k - x^*)  
- 2\gamma_k w_k^T (x_k - x^*) \\
& + \gamma_k^2 \| \nabla F(x_k) \|^2  + \gamma_k^2 \| {w_k} \|^2 +
2 \gamma_k^2  w_k^T \nabla F(x_k).
\end{align*}
Taking expectations conditioned on $\scrF_k$ on both sides \yux{and noticing that $x_k$ is deterministic given $\scrF_k$, and
\begin{align*}
\bE[ w_k^T(x_k - x^*) \mid \scrF_k ] & = (x_k - x^*)^T \bE[ w_k \mid \scrF_k ] = 0 \\
\mbox{ and }  \bE[ w_k^T \nabla F(x_k) \mid \scrF_k ] & =  \nabla F(x_k)^T \bE[ w_k \mid \scrF_k ] = 0,
\end{align*}
we have that in an almost sure sense, for any $r > 0$, }
\begin{align*}
& \quad \bE[ \| x_{k+1} - x^* \|^2 \mid \scrF_k]  
 \leq \| x_k - x^* \|^2 - 2
\gamma_k \nabla F (x_k)^T (x_k - x^*) + \gamma_k^2 \| \nabla F(x_k) \|^2 + \gamma_k^2 \bE [ \| w_k \|^2 \mid \scrF_k ] \\
& \sv{ \ = \ } \| x_k - x^* \|^2 - 2 \gamma_k \sv{( \nabla F (x_k) - \nabla F(x^*) )}^T (x_k - x^*) 
 + \gamma_k^2 \| \nabla F(x_k) - \nabla F(x^*) \|^2 + \gamma_k^2 \bE {[ \| w_k \|^2 \mid \scrF_k ]} \\
& \leq \| x_k - x^* \|^2 - 2 \gamma_k c \| x_k - x^* \|^2 + \gamma_k^2 L^2 \| x_k - x^* \|^2 + \gamma_k^2 ( v_1 \| x_k \|^2 + v_2 ) \\
& \leq \| x_k - x^* \|^2 - 2 \gamma_k c \| x_k - x^* \|^2 + \gamma_k^2 L^2 \| x_k - x^* \|^2 \\
& \yxf{ + \gamma_k^2 \left( v_1 \left( \left(1 + \frac{1}{r} \right) \| x_k - x^* \|^2 + (1+r) \| x^* \|\sv{^2} \right) + v_2 \right) } \\
& = \left(1 - 2 c \gamma_k + \gamma_k^2 \left( L^2 + v_1\left(1 +
				\frac{1}{{r}}\right) \right) \right) \| x_k - x^* \|^2 + \gamma_k^2 \left( v_1(1 + \vvs{r}) \| x^* \|^2 + v_2 \right),
\end{align*}
where $ \nabla F(x^*) = {\bf 0}$ implies the \sv{first equality}, {the
	strong convexity of $F$ and the Lipschitz continuity of $\nabla F$,
		   together \sv{with Ass.~\ref{Ass: uncon-sp} (2(ii))}}
{lead to} the \yux{second} inequality,
	{while for all $a, b$ and $r > 0$, the \yux{third} inequality follows from}
\begin{align} \label{Cauchy-SIADMM}
& \| a+b \|^2 \leq (1+{r}) \| a \|^2 + (1+1/{r}) \| b \|^2.
\end{align}
\noindent (ii) Let $\tau_k \triangleq 2c \gamma_k - \gamma_k^2 (L^2 + v_1(1+1/r))$, $\mu_k \triangleq \gamma_k^2(v_1(1+r)\| x^* \|^2 + v_2) $.  Since $\sum_{k=1}^\infty \gamma_k = +\infty$, $\sum_{k=1}^{\infty} \gamma_k^2 < +\infty$, $c < L$, we have that $\tau_k < 1, \forall k \geq 1$, $\tau_k > 0$ for $k$ large enough, $\sum_{k=1}^\infty \tau_k = +\infty$, $\sum_{k=1}^{\infty} \mu_k < +\infty$, $\lim_k \frac{\mu_k}{\tau_k} = 0$. Therefore, by Lemma~\ref{lm: supermartingale}, 
   \vs{$ \| x_k - x^* \|^2 \xrightarrow{k \to \infty} 0$ in an a.s. sense.}
\end{proof}
\begin{remark}
We observed after the writing of this paper that
a result similar to (ii) was derived in \cite[Prop.~4.1 \vs{and Ex.~4.2}]{Bertsekas:1996:NP:560669}. \vs{Under}
weaker assumptions on  $F$ ($\nabla F$ is Lipschitz continuous but $F$ is not
necessarily convex, $\bE[\| w_k \|^2 \mid \Fscr_k] \leq v_1 \| \nabla F(x_k)
\|^2 + v_2$ \sv{a.s.}) \vs{and a possibly random sequence} $\gamma_k$, \vs{it is shown that $\lim_{k
\rightarrow \infty} \nabla F(x_k) = 0$. Statement (ii) in our result can be
viewed as a corollary of this result when $F(x)$ is strongly convex. However,
(i) does not appear to have been proven and serves as a vital requirement 
for proving \vs{our subsequent asymptotic and rate statements}}.\qed
\end{remark}

 If we take {un}conditional
   expectations on both sides of \eqref{superMatingale-unconsp} and define $e_k \triangleq \bE{[\|x_k - x^* \|^2 ]}$, 
$M \triangleq  L^2 + v_1(1+\frac{1}{r})$,  and $C
\triangleq  v_1(1+r) \| x^* \|^2 + v_2$, then
\begin{align}\label{RecursiveIneq-unconsp}
e_{k+1} \leq (1 - 2c \gamma_k + \gamma_k^2 M) e_k + \gamma_k^2C, \qquad  \forall k \geq 1,
\end{align}
{Next, we derive the rate statement for}  $e_k$ from
\eqref{RecursiveIneq-unconsp}. \yxe{This \vs{result is essential in analyzing (SI-ADMM)} in the next section.}
\begin{theorem} \label{thm: conv-rate-of-meansquare-unconsp}
\yxf{\sv{Suppose} $\{ x_k \}$ \sv{denotes} the sequence generated by \eqref{SA-SI-ADMM}} \yxe{\sv{and} Assumption~\ref{Ass: uncon-sp} holds.} \sv{Let} $\gamma_k = \yxf{\gamma_0}/k$, where $\gamma_0 > \frac{1}{2c}$. Let $K {\triangleq} \lceil \frac{\gamma_0^2M}{2c \gamma_0 - 1}
\rceil + 1$ and $Q{(\gamma_0,K)}  \triangleq \max \left\{ \frac{\gamma_0^2 C}{2c\gamma_0 - \gamma_0^2M/K - 1} , Ke_K \right\}$. Then 
\begin{align} \label{conv-rate-of-meansquare-unconsp}
e_k \leq \frac{{Q(\gamma_0,K)}}{k}, \sv{\qquad \mbox{ for any $k \geq K$}}.
\end{align}
\begin{proof} {We prove the result by \sv{conducting an} induction on $k$.} 
When $k = K$, \eqref{conv-rate-of-meansquare-unconsp} holds {trivially}. Suppose \eqref{conv-rate-of-meansquare-unconsp} holds at arbitrary $k \geq K$, then by \eqref{RecursiveIneq-unconsp}, 
\begin{align*}
e_{k+1} 
& \leq (1 - 2c\gamma_k + \gamma_k^2 M) e_k + \gamma_k^2 C 
 \leq (1 - 2c\gamma_0/k + \gamma_0^2 M/k^2 ) Q{(\gamma_0,K)}/k + \gamma_0^2 C/k^2 \\
& = \left( 1 - 2c \frac{\gamma_0}{k} + \gamma_0^2 \frac{M}{k^2} + \gamma_0^2 \frac{C}{kQ{(\gamma_0,K)}} \right) \frac{Q(\gamma_0,K)}{k}.
\end{align*}
By definition of $Q{(\gamma_0,K)}$, $Q{(\gamma_0,K)}/(\gamma_0^2 C) \geq 1/(2c\gamma_0-\gamma_0^2 M /K -1)$, implying that
\begin{align*}
 e_{k+1} 
& \leq  \left( 1 -\frac{2c\gamma_0}{k} + \frac{\gamma_0^2 M}{k^2} + \frac{2c\gamma_0 - \frac{\gamma_0^2M}{K} - 1}{k} \right) \cdot \frac{ Q{(\gamma_0,K)}}{k}\\
& = \left(1-\frac{1}{k} + \frac{\gamma_0^2M}{k^2} -  \frac{\gamma_0^2M}{(K \cdot k)} \right) \cdot \frac{(k+1)}{k} \cdot \frac{ Q{(\gamma_0,K)}}{(k+1)} \\
& =  \left( \frac{k-1}{k} + \frac{\gamma_0^2M}{k} \cdot \left(\frac{1}{k} - \frac{1}{K} \right) \right) \cdot \frac{(k+1)}{k} \cdot \frac{ Q{(\gamma_0,K)}}{(k+1)} 
 \leq \frac{k-1}{k} \cdot \frac{(k+1)}{k} \cdot \frac{ Q{(\gamma_0,K)}}{(k+1)}  < \frac{Q{(\gamma_0,K)}}{k+1},
\end{align*}
where the \vs{penultimate} inequality \vs{follows from} $k \geq K \Rightarrow 1/k - 1/K \leq 0$.
\end{proof}
\end{theorem}
\begin{remark} \label{Choices of a and b}
To explicitly bound $e_K$ using $e_1$, {we may derive a bound for all $k$ as follows}: 
	Suppose $a_k \triangleq 1 - 2c \gamma_k + \gamma_k^2 M$ for $k \geq
	1$. {It follows that $e_{\usss{k+1}} \leq a_k e_k + \gamma_k^2 C$. By a suitable
	choice of steplength sequence $\{\gamma_k\}$, we may guarantee that
	$a_k > 0$ for all $k \geq 1$,} allowing for constructing a general
	inequality for $e_{k+1}$ given by the following for $k \geq 2.$
	{$$e_{k+1} \leq \left(\prod_{i=1}^k a_i\right)e_1 + C \left(
			\sum_{i=1}^{k-1} \left(\gamma_i^2 \prod_{j=i+1}^k a_j\right)
			+ \gamma_k^2 \right).$$} 
	We may then show that $e_K \leq \hat a e_1 + \hat b C$, where 
$\hat a \triangleq  \prod_{i=1}^{K-1} a_i, K \geq 2$
	and $$\hat b \triangleq 
\begin{cases}
\gamma_1^2, & \mbox{if } K = 2, \\
\sum_{i=1}^{K-2} (\gamma_i^2  a_{i+1} \cdot \hdots \cdot a_{K-1}) +
\gamma_{K-1}^2, & \mbox{if } K \geq 3.
\end{cases}
$$ 
We \vs{denote $(\hat a, \hat b)$ by $(a_x,b_x)$ (for $x$-update) or $(a_y,b_y)$ (for $y$-update) 
in later subsections. } \qed
\end{remark}

\subsection{SI-ADMM: A stochastic inexact ADMM scheme} \label{subsec: SI-ADMM}
Since \eqref{yexactupdate} and \eqref{xexactupdate} in Alg.~\ref{generalized admm}
necessitate exact solutions {impossible to obtain in stochastic
	regimes}, we {propose} a {\em stochastic inexact} \vs{ADMM} ({\bf SI-ADMM}), an
extension of the generalized ADMM scheme, \vs{in which} the sequence of
iterates generated by the scheme are random variables, \vs{each of which requires}
taking a finite (but) increasing number of (stochastic) gradient steps. \yxe{Here $T_{k+1}^x-1$ and $T_{k+1}^y-1$ denote the number of sampled gradients generated within the $x$ and $y$ update to ensure meeting \vs{a} suitable error criterion.} The ({\bf SI-ADMM}) is defined in Alg.~\ref{s-admm}, where
{
\begin{align}
 \tilde{\Lscr}_1 (x, y, \lambda,\xi^x) & \triangleq  \tilde{f}(x, \xi^x) + g(y) - \lambda^T(Ax + By - b)
\label{L1-SIADMM}
+ \frac{\rho}{2} \| Ax + By - b \|^2,\\
 \tilde{\Lscr}_2 (x, y, \lambda,\xi^y) & \triangleq  f(x) + \tilde{g}(y,\xi^y) - \lambda^T(Ax + By - b)
\label{L2-SIADMM}
+ \frac{\rho}{2} \| Ax + By - b \|^2.
\end{align}
}
\begin{algorithm}
\caption{{\bf SI-ADMM:} A stoch. inexact ADMM scheme}
\label{s-admm}
\begin{enumerate}
\item[(0)] Choose $Q$, $P$ and sequences $\{ T_{k+1}^y, T_{k+1}^x\}_{k \geq 0}$ as number of samples generated for updates \eqref{Ty-SIADMM} and \eqref{Tx-SIADMM}; Given positive scalars $\rho, \gamma, \gamma_x, \gamma_y$ and initial points $x_0, y_0,\lambda_0$; Let $k := 0$;
\item[(1)] Let $x_{k+1}, y_{k+1}, \lambda_{k+1}$ be given by the following:
\begin{align} 
\notag
	& y_{k,j+1} := y_{k,j} - \frac{\gamma_{y}}{j}[\nabla_y  \tilde{\mathcal{L}}_{2}(x_k,y_{k,j},\lambda_k,\xi^y_{k,j}) + Q(y_{k,j} - y_k) ],\\
\notag
	& \qquad \qquad \  j = 1, \hdots, T_{k+1}^y -1, y_{k,1} := y_k ,\\
&	\pmb{y_{k+1} := y_{k, T_{k+1}^y}} \label{Ty-SIADMM} \\
\notag
	& x_{k,j+1} := x_{k,j} - \frac{\gamma_{x}}{j} [\nabla_x \tilde{\mathcal{L}}_{1}(x_{k,j},y_{k+1},\lambda_k,\xi^x_{k,j}) + P(x_{k,j} - x_k) ], \\
\notag
     & \qquad \qquad \  j = 1, \hdots,T_{k+1}^x -1,  x_{k,1} := x_k, \\
&   \pmb{x_{k+1} := x_{k, T_{k+1}^x}} \label{Tx-SIADMM} \\
& \lambda_{k+1} := \lambda_k - \gamma \rho( Ax_{k+1} + By_{k+1} - b). 
\end{align}
\item [(2)] \yxf{$k:=k+1$, return to (1).}
\end{enumerate}
\end{algorithm}
{$y_{k+1}$ and $x_{k+1}$ can be viewed as inexact solutions to} \eqref{yexactupdate} and \eqref{xexactupdate} in Alg.~\ref{generalized admm}, respectively, obtained through a standard stochastic approximation \sv{scheme}. Therefore, the sequence $u_{k+1} = (x_{k+1}; y_{k+1}; \lambda_{k+1})$, $\forall k \geq 0$ is also inexact.
\yxe{\sv{Suppose the sequence $\{ \xi^y_{k,1}, \hdots, \xi^y_{k,T^y_{k+1}-1}, \xi^x_{k,1}, \hdots, \xi^x_{k,T^x_{k+1}-1}\}_{k \geq 0}$}
 \vs{represents the sequence of realizations} of $\xi$. We denote the history of the algorithm as follows. First, $\Fscr_0 \triangleq
		\{x_0,y_0,\lambda_0\}$. Then at the \uvs{$(k+1)$}th epoch ($k \geq 0$),
		$\Fscr_{\uss{k+1}}^y \triangleq \Fscr_{\uss{k}} \cup
			\{\xi^y_{k,1}, \hdots, \xi^y_{k,T_{k+1}^y-1} \}$,
		$\Fscr_{\uss{k+1}} \triangleq \Fscr_{{\uss{k+1}}}^y \cup \{\xi^x_{k,1}, \hdots, \xi^x_{k,T_{k+1}^x-1} \}$.  }
Moreover, we may define a map
$\Gamma_k(u)$, akin to the map $\Gamma^*(u)$ specified in
\eqref{def-Gamma-SI-ADMM},
\begin{align}\label{def-Gamma-k-SIADMM}
	\Gamma_k(u) \uvs{ \ \triangleq \ } \pmat{ y_{\uss{k}}^{\rm s-inex}(x,y,\lambda) \\
						 x_{\uss{k}}^{\rm s-inex}(x, \tilde y_{\uss{k}}^{\rm s-inex}, \lambda) \\
						 \lambda_{\uss{k}}^{\rm s-inex}(\tilde x_{\uss{k}}^{\rm s-inex}, \tilde y_{\uss{k}}^{\rm s-inex}, \lambda)},
\end{align}
{where $u = (x;y;\lambda)$, $y_k^{\rm s-inex}, x_k^{\rm s-inex}, \lambda_k^{\rm s-inex}$ are three maps such that $y_k = y_{\uss{k}}^{\rm s-inex}(x_{k-1},y_{k-1},\lambda_{k-1})$, $ x_k = x_k^{\rm s-inex}(x_{k-1}, y_k,\lambda_{k-1})$, $\lambda_k = \lambda_k^{\rm s-inex}(x_k, y_k, \lambda_{k-1})$. Moreover, $\tilde y_{\uss{k}}^{\rm s-inex} \triangleq y_{\uss{k}}^{\rm s-inex}(x,y,\lambda)$, $\tilde x_k^{\rm s-inex} \triangleq x_k^{\rm s-inex}(x, \tilde y_{\uss{k}}^{\rm s-inex},\lambda)$. Consequently, $\Gamma_k(u)$
	is a single-valued random map.} 
Its image depends on both the input and samples. 
	To ensure convergence of {$\{ u_k \}$}, we \vvs{assume the following}. 
	\begin{assumption} \label{Ass: Strong convexity} \vs{The} functions $f$ and $g$ are both
	{strongly} convex in their respective arguments with constants
	$\mu_f$ and $\sigma_g$, respectively. {In addition, $f$ and $g$ are both
	continuously differentiable with Lipschitz continuous gradients with
	constants $L_f$ and $L_g$, respectively.}
\end{assumption}
\begin{assumption}\label{Ass: Full Row Rank}  
The matrix $[A\quad B]$ has full row rank.
\end{assumption}

\begin{remark} \label{Rmk: Uniqueness of KKT point}
Assumptions \ref{Ass: existence of KKT point-SI-ADMM}, \ref{Ass: Strong convexity}, and \ref{Ass: Full Row Rank} are sufficient to claim that
there {exists a unique} triple {$(x^*;y^*;\lambda^*) = u^* $} that {satisfies} the KKT conditions:
 $A^T \lambda^*  = \nabla f(x^*),
 B^T \lambda^*  = \nabla g(y^*),  \mbox{ and }
 Ax^* + By^*  = b.$ 
Since $f$ and $g$ are {strongly} convex, \eqref{SOpt} has a unique
{primal} optimal solution. {Furthermore, since the KKT conditions are necessary and
sufficient, they} {admit} a unique  {solution tuple} 
$(x^*,y^*)$. Since $[A\quad B]$ has {linearly independent rows,} $\lambda^*$ is
uniquely determined by $(x^*,y^*)$. Since the KKT point $u^*$ is
unique and the \sv{contractive} property in Theorem \ref{linear convergence of
g-ADMM} holds for the same $u^*$ regardless of initial point\sv{;} i.e., inequality
\eqref{contraction of iterates} holds for every $u_k$ for the same $u^*$.
Moreover, Assumption \ref{Ass: Strong convexity} implies Assumption \ref{Ass: objective function-SI-ADMM}(d), \sv{while} Assumption
\ref{Ass: Full Row Rank} can be made without loss of generality. \qed
\end{remark}

\subsection{Error bounds and choice of $T_{k+1}^y$ and $T_{k+1}^x$}\label{subsec: choice of sample size}
\yxe{In this section, we will present some error bounds and \vs{derive a}
recursive inequality (Lemma~\ref{lm: main recursive inequality}) important in
proving the main result in Section~\ref{sec: convergence and rate-SIADMM}}. {We
begin by providing \yux{a necessary moment assumption}. Throughout the remainder of the paper, we assume that the \sv{sequence} \\
\sv{$\{\xi^y_{k,1}, \hdots, \xi^y_{k,T^y_{k+1}-1}, \xi^x_{k,1}, \hdots, \xi^x_{k,T^x_{k+1}-1}\}_{k \geq 0}$}  generated by Algorithm~\ref{s-admm} is i.i.d. 
 
\begin{assumption}\label{Ass: ConditionalUnbiased} 
\yux{
For any given $x \in \bR^n, y \in \bR^m, \xi \in \Xi$, 
let $w^x \triangleq \nabla_x \tilde f(x,\xi) - \nabla_x f(x)$ and $w^y \triangleq \nabla_y \tilde g(y,\xi) - \nabla_y g(y)$. For any $x \in \bR^n, y \in \bR^m$, 
$\mathbb{E}[w^x] = {\bf 0}$, $\mathbb{E}[w^y] =
{\bf 0}$. Furthermore, there exists constants $v_1^x, v_2^x, v_1^y, v_2^y$, such that for any $x \in \bR^n, y \in \bR^m$,  $\mathbb{E}[\|w^x \|^2] \leq v_1^x \|x\|^2 +
v^x_2$ and $\mathbb{E}[\|w^y \|^2] \leq v_1^y \|y\|^2 +
v^y_2$.} 
\end{assumption}
Suppose {$\bf c_y$} and {$\bf c_x$} {denote the
		 {strong} convexity} constants of $\mathcal L_\rho(x,y,\lambda)
		 + \frac{1}{2} \| y - z \|_Q^2$ and $\mathcal L_\rho(x,y,\lambda) +
		 \frac{1}{2} \| x - z \|_P^2$ in
		 terms of $y$ {(uniformly in $x$)}  and $x$ {(uniformly in
				 $y$)}, respectively, where $z \in \bR^n$ can be taken
		 as any vector. {Similarly, suppose} {$\bf L_y$} and {$\bf L_x$}
		 denote the Lipschitz constants of $\nabla_y \mathcal
		 L_A(x,y,\lambda) + Q(y-z)$ and $\nabla_x \mathcal
		 L_A(x,y,\lambda) + P(x-z)$ in terms of $y$ and $x$,
	 respectively, where $z \in \bR^n$ again can be taken as any vector.
		 Note that $c_y, c_x, L_y, L_x$ are {assumed to be} independent of $x$, $y$,
	 $\lambda$ {and} $z$.  {From definitions \eqref{L1-SIADMM} \sv{and} \eqref{L2-SIADMM}, the following equations hold.}
\begin{align*} 
& \nabla_x \tilde{\Lscr}_1 (x, y, \lambda, \xi^x) + P(x - z_1)
- \nabla_x \Lscr_\rho (x, y, \lambda) - P(x - z_1) 
 = \underbrace{\nabla_x \tilde{f}(x,\xi^x) - \nabla_x f(x)}_{\vs{\triangleq  w^x}}, \forall x,y,\lambda, z_1,\\
& \nabla_y \tilde{\Lscr}_2 (x, y, \lambda,\xi^y) + Q(y-z_2) - \nabla_y
\Lscr_\rho (x, y, \lambda) - Q(y-z_2) 
 = \underbrace{\nabla_y \tilde{g}(y, \xi^y) -
\nabla_y g(y)}_{\vs{\triangleq  w^y}},  \forall x,y,\lambda, z_2.
\end{align*} 
In fact, $\nabla_x \tilde{\Lscr}_1 (x, y, \lambda,\xi^x) + P(x-z_1)$ and
$\nabla_y \tilde{\Lscr}_2 (x, y, \lambda,\xi^y) + Q(y - z_2)$ are the
stochastic gradients utilized in the  inexact updates
\vs{\eqref{Ty-SIADMM}--\eqref{Tx-SIADMM}}  of Algorithm~\ref{s-admm},
respectively, where $z_1$ and $z_2$ are taken as the iterates $x,y$ of last
iteration. Next, {we derive bounds for  updates \eqref{Ty-SIADMM} and
\eqref{Tx-SIADMM} in Lemma~\ref{Lm: errorbound-SIADMM} by recalling that both involve} solving unconstrained
stochastic approximation {problems}. \yxe{Therefore, the proof follows \vs{in a similar fashion to that
in} Theorem~\ref{thm: conv-rate-of-meansquare-unconsp}. \vs{In fact},  
the bounds derived for the subproblems will be used to \vs{obtain a} 
bound for the major iterates \vs{generated by Algorithm~\ref{s-admm}} (Lemma~\ref{lm: main recursive inequality}).}
Before presenting the Lemma, let us define $y_{k+1}^*, x_{k+1}^*, \tilde
x_{k+1}^*, \lambda_{k+1}^*,u_{k+1}^*$ as follows:
\begin{align}
\label{y_k_star-SIADMM}
& y_{k+1}^* \triangleq \mbox{arg}\hspace{-0.02in}\min_{y} \{ \Lscr_\rho(x_k,y,\lambda_k) + \frac{1}{2}  \| y-y_k \|^2_Q \}, \\
\label{x_k_star-SIADMM}
& x_{k+1}^* \triangleq \mbox{arg}\hspace{-0.02in}\min_{x} \{ \Lscr_\rho(x,y_{k+1}^*,\lambda_k) + \frac{1}{2}\| x-x_k \|^2_P \}, \\
\label{x_tilde_k_star-SIADMM}
& \tilde{x}_{k+1}^* \triangleq \mbox{arg}\hspace{-0.02in}\min_{x} \{ \Lscr_\rho(x,y_{k+1},\lambda_k) + \frac{1}{2}\|x-x_k\|^2_P \}, \\
\label{lambda_k_star-SIADMM}
& \lambda_{k+1}^* \triangleq \lambda_k - \gamma \rho (Ax_{k+1}^* + By_{k+1}^* - b), \\
\label{u_k_star-SIADMM}
& {u^*_{k+1}} \triangleq (x_{k+1}^*; y_{k+1}^*; \lambda_{k+1}^*) = \Gamma^*((x_k;y_k;\lambda_k)).
\end{align}
\sv{Furthermore,} the parameters $C_y^{k+1}, M_y, K_y, C_x^{k+1}, M_x, K_x$ \sv{are defined as follows.}
\begin{align}
\label{Def: M_y and K_y}
& C_y^{(k+1)} \triangleq v_1^y(1+R)\| y_{k+1}^* \|^2 + v_2^y,
\quad M_y \triangleq L_y^2 + v_1^y \left(1 + \frac{1}{R} \right), \quad
K_y \triangleq \lceil \frac{\gamma_y^2M_y}{2c_y\gamma_y - 1} \rceil + 1, \\
\label{Def: M_x and K_x}
& C_x^{(k+1)} \triangleq v_1^x(1+R)\|\vvs{\tilde x}_{k+1}^* \|^2 + v_2^x, \quad M_x \triangleq L_x^2 + v_1^x \left(1 + \frac{1}{R} \right), \quad
K_x \triangleq \lceil \frac{\gamma_x^2M_x}{2c_x\gamma_x - 1} \rceil + 1. 
\end{align}
{\begin{lemma}\label{Lm: errorbound-SIADMM}
Consider Algorithm~\ref{s-admm}. Suppose \vs{Assumptions}~\ref{Ass: Strong convexity} and \ref{Ass: ConditionalUnbiased} hold. {Then \vs{for} any $R > 0$, 
the following holds \vvs{a.s.} for all $k \geq 0$ for updates \eqref{Ty-SIADMM} and \eqref{Tx-SIADMM}.} 
\begin{align}
\label{y-errorbound}
 \bE[ \| y_{k+1} - y_{k+1}^* \|^2 \mid \scrF_k ]  & \leq \frac{\theta^{k+1}_y}{T_{k+1}^y}, \mbox{if } T_{k+1}^y \geq K_y, \\
\label{x-errorbound}
 \bE[ \| x_{k+1} - \tilde x_{k+1}^* \|^2 \mid \scrF_{k+1}^y ] 
 & \leq \frac{\theta^{k+1}_x}{T_{k+1}^x}, \mbox{if } T_{k+1}^x \geq K_x,
\end{align}
where
\begin{align}
& \notag \theta^{k+1}_y \triangleq \max \left\{ \frac{\gamma_y^2 C_y^{(k+1)}}{2 c_y \gamma_y - \gamma_y^2 M_y/K_y - 1}, K_y \bar{e}_{K_y}^{y,k+1} \right\}, 
 \notag \theta^{k+1}_x \triangleq \max \left\{\frac{\gamma_x^2 C_x^{(k+1)}}{2 c_x \gamma_x - \gamma_x^2 M_x/K_x - 1}, K_x \bar{e}_{K_x}^{x,k+1} \right\},  
\end{align}
$\bar{e}_{K_y}^{y,k+1} = \bE[\| y_{k,K_y} - y_{k+1}^* \|^2 \mid \yxe{\Fscr_k}] \leq a_y \| y_k - y_{k+1}^* \|^2 + b_yC_y^{(k+1)}$ a.s.,\\ 
$\bar{e}_{K_x}^{x,k+1} = \bE[\| x_{k,K_x} - \tilde{x}_{k+1}^* \|^2 \mid \yxe{\Fscr_{k+1}^y} ] \leq a_x \| x_k - \tilde
x_{k+1}^* \|^2 + b_xC_x^{(k+1)}$ a.s., \\
{and $a_y, b_y, a_x, b_x$ are {constants determined by problem parameters $c_y, c_x, \gamma_y, \gamma_x, L_y, L_x$, and $R$.}}
 \end{lemma}
{\begin{proof} {Omitted} (\yxe{\vs{Similar} analysis to Th.~\ref{thm: conv-rate-of-meansquare-unconsp}}) \vs{while} values of $a_y, b_y, a_x, b_x$ are discussed in Rem.~\ref{Choices of a and b}.
\end{proof}
}

\yxf{Note that if assumptions of Theorem~\ref{linear convergence of g-ADMM} hold and \eqref{SOpt} has \vs{a} unique KKT point $u^*$, then based on Theorem~\ref{linear convergence of g-ADMM}, there exists $\delta > 0$, such that for any $k \geq 0$,
\begin{align} \label{contraction of iterates-SIADMM}
\|u_{k+1}^* - u^*\|_G^2 \leq \frac{1}{1+\delta} \|u_k - u^*\|_G^2, \quad 
\end{align}
where 
$G$ is defined in \eqref{Def: G}.
}
\yxe{\vvs{We now} {present} \vs{a} recursive inequality essential in analyzing \vs{a.s.} convergence and the convergence rate of the major iterates $x_k,y_k,\lambda_k$ in Algorithm~\ref{s-admm}}. We continue to use the {constants defined in \eqref{Def: M_y and K_y}, \eqref{Def: M_x and K_x}.
\begin{lemma} \label{lm: main recursive inequality}
Suppose Assumptions~\ref{Ass: existence of KKT point-SI-ADMM},
\ref{Ass: P&Q-Si-ADMM}(i), \ref{Ass: Strong convexity}, \ref{Ass: Full Row
Rank}, and \ref{Ass: ConditionalUnbiased} hold. In addition, suppose $\gamma$ and $P$
satisfy \eqref{gammaP}, $\eta \in (0,1)$, and $T > 0$. \sv{Then the following hold.} \\
(i) \yxf{\eqref{SOpt} has an unique KKT point $u^*$  and there exists  $\delta > 0$ such that 
\eqref{contraction of iterates-SIADMM} holds for any $k \geq 0$.} \\
(ii) For any $k \geq 0$, if $T_{k+1}^y \geq K_y, T_{k+1}^x \geq K_x$, then for any $R_{0,k} > 0$, 
\begin{align} \label{supermartigale inequality 0}
 \bE[\| u_{k+1} - u^* \|^2_G \mid \scrF_k ]  & \leq \left[ (1+R_{0,k}) \left( \frac{\zeta_2^x}{T_{k+1}^x} + \frac{\zeta_2^y}{T_{k+1}^y} + \frac{\zeta_2^{xy} }{T_{k+1}^x T_{k+1}^y} \right) + \frac{1 + 1/R_{0,k}}{1 + \delta} \right] \| u_k - u^* \|_G^2 \ \notag \\
& + (1+R_{0,k}) \left( \frac{\zeta_1^x}{T_{k+1}^x} + \frac{\zeta_1^y}{T_{k+1}^y} + \frac{\zeta_1^{xy}}{T_{k+1}^x T_{k+1}^y } \right),
\end{align}
\yxf{where $\zeta_1^x, \zeta_1^y, \zeta_1^{xy},  \zeta_2^x, \zeta_2^y, \zeta_2^{xy}$ are \sv{suitably defined nonnegative constants.}}  

\noindent (iii) Suppose $R_{0,k} \equiv R_0$, $a \triangleq \frac{1 + 1/R_{0}}{1 + \delta}$, $C_1 \triangleq
(1 + R_{0})(\zeta_2^x + \zeta_2^y )/T$, $C_2 \triangleq (1
+ R_{0}) \zeta_2^{xy} / T^2$, $C_3 \triangleq (1+R_{0})( \zeta_1^x + \zeta_1^y )/T$, $C_4 \triangleq (1 + R_{0}) \zeta_1^{xy} /T^2,T_{k+1}^y {\triangleq} \max \left\{K_y, \lceil T/\eta^k \rceil \right\}$, $ T_{k+1}^x {\triangleq} \max \left\{ K_x, \lceil T/\eta^k \rceil \right\}$. Then \vs{the following holds a.s. for any index $k$.}
\begin{align} \label{supermartigale inequality 0.5}
  \bE[\| u_{k+1} - u^* \|^2_G \mid \scrF_k ] 
  \leq \left[ \left( C_1 + C_2 \eta^k \right) \eta^k + a \right] 
	 \| u_k - u^* \|_G^2   + \left( C_3+ C_4 \eta^k \right) \eta^k.
\end{align}
\end{lemma}
\begin{proof}
\noindent (i) See Remark~\ref{Rmk: Uniqueness of KKT point}. Note that Theorem~\ref{linear convergence of g-ADMM} can be applied. 
\noindent (ii) See Appendix.
\noindent (iii) \vs{By substituting $T_{k+1}^y$ and $T_{k+1}^x$ by $\tfrac{T}{\eta^k}$} \vs{and} $R_{0,k}$ by $R_0$ in \eqref{supermartigale inequality 0}, the result follows.
\end{proof}
\yxf{We now refine this result for the setting where $g$ is simple {and the $y$ update can be obtained exactly in Algorithm~\ref{s-admm}}, i.e., update \eqref{Ty-SIADMM} is modified to 
\begin{align} \label{Ty-SIADMM-modified}
y_{k+1} := \mbox{arg}\hspace{-0.02in}\min_{y} \ \Lscr_{\rho}(x_k,y,\lambda_k) + \sv{1\over 2}\| y - y_k \|_Q^2,
\end{align}
where $\Lscr_{\rho}(x,y,\lambda)$  \sv{denotes the} augmented Lagrangian \sv{function} defined in
\eqref{Def: ALfunc}.} \vs{Consequently, the \yux{moment} assumption
(Assumption~\ref{Ass: ConditionalUnbiased}), while imposed directly, pertains
only to the sample $\nabla_x {\tilde f} (x,\xi)$}. \yxf{Since $y_{k+1}$ can be
exactly obtained, we let $\scrF_{k+1}^y = \scrF_k$.
Note that the following corollary \vs{finds relevance} 
\vs{in the numerics}.} 
\begin{corollary} \label{Corr: theobound}
\sv{Consider Algorithm~\ref{s-admm} where update \eqref{Ty-SIADMM} is modified to \eqref{Ty-SIADMM-modified}} \yux{($g$ is simple such that \eqref{Ty-SIADMM-modified} can be computed exactly). Suppose Assumptions~\ref{Ass: existence of KKT point-SI-ADMM}, \ref{Ass: objective function-SI-ADMM} (a),
\ref{Ass: P&Q-Si-ADMM}, and \ref{Ass: ConditionalUnbiased} hold and $\gamma$, $P$ satisfy \eqref{gammaP}. Moreover, suppose that \eqref{SOpt} has an unique KKT point $u^*$.} \sv{Then the following hold.}\\
(i) There exists  $\delta > 0$ such that 
\eqref{contraction of iterates-SIADMM} holds for any $k \geq 0$. \\
(ii)
If $T_{k+1}^x \geq K_x$ \vs{in} Algorithm~\ref{s-admm} \vs{and  $R_{0,k} > 0$},  then the following holds \vs{a.s.} for any $k \geq 0$.
\yxf{\begin{align} \label{theobound-supermartingale}
 \bE[ \| u_{k+1} - u^* \|_G^2 \mid \Fscr_k] \leq \left( \frac{ (1+R_{0,k}) \zeta_2}{T_{k+1}^x} + \frac{1 + 1/R_{0,k}}{ 1 + \delta} \right) \|u_k - u^*\|_G^2 + \frac{ (1+R_{0,k}) \zeta_1}{T_{k+1}^x}, 
\end{align}
where $\zeta_1, \zeta_2$ are \sv{suitably defined nonnegative constants.}} 
\end{corollary}
\begin{proof}
See Appendix.
\end{proof} 

\section{Convergence and rate analysis}\label{sec: convergence and rate-SIADMM}
In this section, we examine the convergence of the (random) sequence
$\{u_k\}$ generated by ({\bf SI-ADMM}), both from an asymptotic  and a
rate  standpoint, based on the choice of the sequences
$\{T_{k+1}^y\}$ and $\{ T_{k+1}^x \}$. 
The sequence $\{u_k\}$ created by Algorithm~\ref{s-admm} has the following relationship:
\begin{align} u_{k+1} := \Gamma_{k+1} (u_k), \label{def-uk-SIADMM}\end{align}
where $\Gamma_k(\bullet)$ is defined in \eqref{def-Gamma-k-SIADMM}. Based on the supermartingale convergence lemma (Lemma~\ref{lm: supermartingale}), we obtain a.s. 
convergence if $\{ T_{k+1}^y \}$ and $\{ T_{k+1}^x \}$ are {suitably}
chosen.
\begin{theorem}[{\bf a.s. convergence}]\label{asconv-SIADMM}
\yxf{Consider Algorithm~\ref{s-admm}.} \yux{Suppose Assumptions~\ref{Ass: existence of KKT point-SI-ADMM}, \ref{Ass: P&Q-Si-ADMM}(i), \ref{Ass: Strong convexity}, \ref{Ass: Full Row Rank}, and \ref{Ass: ConditionalUnbiased} hold and $\gamma$ and $P$ satisfy \eqref{gammaP}.} Suppose $T_{k+1}^y$ and
			$T_{k+1}^x$ satisfy 
			$T_{k+1}^y \geq K_y$, $T_{k+1}^x \geq K_x$, $\sum_{k=0}^\infty
			\frac{1}{T_{k+1}^y} < \infty, \sum_{k=0}^\infty
			\frac{1}{T_{k+1}^x} < \infty,$ and $\bE[ \| u_0 - u^*
			\|_G^2] < \infty $. Then $ \| u_k - u^* \|^2_G \rightarrow 0 $ a.s. as $k \rightarrow	\infty$. 
\end{theorem}
\begin{proof} Under the \vs{assumptions of this  Theorem}, Lemma~\ref{lm: main recursive inequality} \vs{may} be applied. Specifically, {we utilize} \eqref{supermartigale inequality 0} and
Lemma~\ref{lm: supermartingale} to prove this result. Let $R_{0,k} \equiv R_0$ \sv{ for all $k$}  in \eqref{supermartigale inequality 0}, \sv{which may be rewritten as follows.}
\begin{align}\label{supermartingale-thm: a.s.convergence}
\bE[ v_{k+1} \mid \Fscr_k] \leq (1 - U_k) v_k + V_k, \quad a.s. \quad \forall k \geq 0, 
\end{align}
\yxf{where $v_k \triangleq \| u_k - u^* \|_G^2$,
$V_k \triangleq (1+R_0) \left( \frac{\zeta_1^x}{T_{k+1}^x} + \frac{\zeta_1^y}{T_{k+1}^y} + \frac{\zeta_1^{xy}}{T_{k+1}^x T_{k+1}^y} \right)$, \vs{and}  $$ U_k \triangleq 1 - \left( \frac{1 + 1/R_0}{1 + \delta} + (1+R_0) \left(  \frac{\zeta_2^x}{T_{k+1}^x} + \frac{\zeta_2^y}{T_{k+1}^y} + \frac{\zeta_2^{xy}}{T_{k+1}^x T_{k+1}^y} \right) \right).$$}
By assumption, $T_{k+1}^y \rightarrow +\infty, T_{k+1}^x \rightarrow +\infty$, implying that $1 - U_k \xrightarrow[]{k \rightarrow \infty} \frac{1 + 1/R_0}{1 + \delta}$ for any fixed $R_0$. Since \eqref{supermartingale-thm: a.s.convergence} holds for any $R_0$, fix $R_0 > 0$ such that $a \triangleq \frac{1 + 1/R_0}{1 + \delta} < 1$. Then $\exists K_0, s.t.\ \forall k \geq K_0$, $1 - U_k < q < 1 \Rightarrow 1-q < U_k < 1, \forall k \geq K_0$. Therefore, the shifted recursion can be stated as follows.
\begin{align}
\notag
& \bE[ v_{k+1} \mid \yxb{v_{K_0},\hdots,v_k}] \leq (1 - U_k) v_k + V_k, \ a.s. , \forall k \geq K_0,
\end{align}
\vs{where  $1-q < U_k < 1, V_k \geq 0$ for all $k \geq K_0.$ If it can be
shown} that $\mathbb{E}[v_{K_0}] < \infty$, $\sum_{k = K_0}^\infty U_k  =
\infty$, $\sum_{k = K_0}^\infty V_k < \infty $ and $\lim_{k \rightarrow \infty}
\frac{V_k}{U_k} =  0$, then by Lemma~\ref{lm: supermartingale}, the \vs{result
follows.} In fact, $\mathbb{E}[v_0] < \infty$ and $\bE[ v_{k+1} ] \leq (1 -
U_k) \bE[ v_k ] + V_k, \forall k \geq 0$ (obtained by taking expectations on
both sides of \eqref{supermartingale-thm: a.s.convergence}) imply
$\mathbb{E}[v_{K_0}] < \infty$. $1-q < U_k < 1, \forall k \geq K_0$ implies
$\sum_{k = K_0}^\infty U_k  = \infty$. Eventually, $\sum_{k=0}^\infty
\frac{1}{T_{k+1}^y} < \infty, \sum_{k=0}^\infty \frac{1}{T_{k+1}^x} < \infty,
U_k > 1-q, \forall k \geq K_0$ imply $\sum_{k = K_0}^\infty V_k < \infty $ and
$\lim_{k \rightarrow \infty} \frac{V_k}{U_k} =  0$.  
 \end{proof}

Next, we consider the case when $T_{k+1}^y$ and $T_{k+1}^x$
increase geometrically and prove \yxf{linear convergence of the major iterates}. \vs{We also show that} the \yxe{overall iteration (sample) complexity \vs{to compute an $\epsilon$-solution is $\mathcal{O}(1/\epsilon)$, \vs{identical to the  canonical iteration and sample complexity}}.} 
{\begin{theorem}[{\bf Rate and overall iteration (sample) complexity under geometric growth of
	$T_{k+1}^x$ and $T_{k+1}^y$}] \label{thm: geometric decay -SIADMM}
Consider Algorithm~\ref{s-admm}. \yux{Suppose Assumptions~\ref{Ass: existence of KKT point-SI-ADMM}, \ref{Ass: P&Q-Si-ADMM}(i), \ref{Ass: Strong convexity}, \ref{Ass: Full Row Rank}, and \ref{Ass: ConditionalUnbiased} hold. In addition, $\gamma$ and $P$ satisfy \eqref{gammaP}.} Suppose $T_{k+1}^y \vvs{\triangleq} \max \left\{K_y, \lceil T/\eta^k \rceil \right\}$, $T_{k+1}^x \vvs{\triangleq} \max \left\{ K_x, \lceil T/\eta^k \rceil \right\}$, $\forall k \geq 0$, where $0< \eta < 1$, $T > 0$. Then  the following statements hold:\\
(i) $\mathbb{E}[\|u_k-u^*\|_G^2] \to 0$ as $k \to \infty$ geometrically.\\
(ii) Suppose \eqref{SOpt} is resolved to  precision $\epsilon \leq 1/e$ by Algorithm~\ref{s-admm}, \ie  $\bE[\|u_k - u^*\|^2_G] \leq \epsilon$. Assume that $\eta > \frac{1}{1+\delta}$. Then the overall iteration (sample) complexity \yxf{to \vs{obtain an $\epsilon$-solution} } is $O(1/\epsilon)$.
\end{theorem}
}
\begin{proof}
(i) \yxf{Under these assumptions, we could apply inequality \eqref{supermartigale inequality 0.5} from Lemma~\ref{lm: main recursive inequality}(ii), where $R_0$ is chosen such that $a \triangleq \frac{1+1/R_0}{1+\delta} < 1$.} 
Denote $r_k = \bE[\| u_k - u^* \|_G^2 ]$. Then by taking expectations on both sides of \eqref{supermartigale inequality 0.5} and by \vs{utilizing the recursion}, we have \sv{for $k \geq 2$} 
\begin{align}  
&   r_{k+1}   \leq \left( \left( C_1 + C_2 \eta^k \right) \eta^k + a \right) r_k + \left( C_3+ C_4 \eta^k \right) \eta^k
 \label{expectationbound}
 \implies r_k   \leq\overbrace{ \left( \prod_{i =0}^{k-1} (C_1 \eta^i + C_2 \eta^{2i} + a ) \right)}^{\vs{\triangleq \beta_0}} r_0 \\
\notag
& + \overbrace{\sum_{i = 0}^{k-2} \big[ (C_3 \eta^i + C_4 \eta^{2i})(C_1 \eta^{i+1} + C_2 \eta^{2(i+1)} + a) {\times} \sv{\cdots} {\times} (C_1 \eta^{k-1} + C_2 \eta^{2(k-1)} + a) \big] + C_3 \eta^{k-1} + C_4 \eta^{2(k-1)}}^{\vs{\triangleq \beta_1}},
\end{align}
 The coefficient \vs{of} $r_0$ \vs{may be bounded as follows.} 
\begin{align*}
 &\vs{\beta_0} \leq \left( \frac{\sum_{i = 0}^{k-1} (C_1 \eta^i + C_2 \eta^{2i} + a ) }{k} \right)^k 
 \leq \left( a + \frac{C_1}{k} \cdot \frac{1-\eta^k}{1-\eta} + \frac{C_2}{k} \cdot \frac{1 - \eta^{2k}}{1 - \eta^2} \right)^k\\
   & \leq \left( a + \frac{C_1}{k} \cdot \frac{1}{1-\eta} + \frac{C_2}{k} \cdot \frac{1}{1 - \eta^2} \right)^k 
 = a^k  \left( 1 + \frac{C_1(1+\eta) + C_2}{a(1-\eta^2)} \cdot \frac{1}{k} \right)^k 
 \leq a^k  e^{\alpha}, 
\end{align*}
where {$\alpha \triangleq  \frac{C_1(1+\eta) +
	C_2}{a(1-\eta^2)}$, the first inequality follows from \sv{noting} that \sv{the}
	geometric mean is less than  \sv{the} arithmetic mean, while the last
	inequality \sv{is a consequence of noting that the following sequence of inequalities holds} for $x \geq 0$:} 
$1 + x \leq (1 + x/2)^2 \leq \cdots \leq (1 + x/k)^k \leq \hdots \leq e^x.$
Let $C_{1,2} = C_1 + C_2$, $C_{3,4} = C_3 + C_4$. Then \yxf{by similar techniques}, \usss{\vs{$\beta_1$} in \eqref{expectationbound} can be bounded as follows:}
\begin{align*}
& \vs{\beta_1}  \leq \sum_{i = 0}^{k-2} \left( (C_{3,4} \eta^i)(C_{1,2} \eta^{i+1}  + a) \hdots (C_{1,2} \eta^{k-1} + a) \right) + C_{3,4} \eta^{k-1} \\
& \leq C_{3,4} \left( \sum_{i=0}^{k-2} \eta^i \left( \frac{\sum_{j = i+1}^{k-1} (a + C_{1,2} \eta^j)}{k-i-1} \right)^{k-i-1} + \eta^{k-1} \right)\\ 
 & =  C_{3,4} \left( \sum_{i=0}^{k-2} \eta^i \left( a + \frac{C_{1,2} \eta^{i+1}}{k-i-1} \cdot \frac{1 - \eta^{k-i-1}}{1-\eta} \right)^{k-i-1} + \eta^{k-1} \right) \\
& \leq C_{3,4} \left( \sum_{i=0}^{k-2} \eta^i a^{k-i-1} \left( 1 + \frac{C_{1,2} \eta}{a(1-\eta)} \cdot \frac{1}{k-i-1} \right)^{k-i-1} + \eta^{k-1} \right) 
 \leq C_{3,4} \left( \sum_{i=0}^{k-2} \eta^i a^{k-i-1} e^{\beta} + \eta^{k-1} \right),
\end{align*}
where $\beta = \frac{C_{1,2} \eta}{a(1-\eta)}$. Let $\tau \triangleq \max\{a, \eta\}$ and $ q \in (\tau,1)$. Then \eqref{expectationbound} \vs{reduces to the following for $k \geq 2$.}
\begin{align}\label{geometric decay}
r_{k} & \leq a^k e^\alpha r_0 + C_{3,4} \left[ \sum_{i=0}^{k-2} \eta^i a^{k-i-1} e^{\beta} + \eta^{k-1} \right] 
 \leq (ae^\alpha r_0  + C_{3,4} e^\beta (k-1) + C_{3,4}) \tau^{k-1}.
\end{align}
We know that $k\tau^k \leq D{(q)}q^k$~\cite{hesamdiss}, implying that $r_k \rightarrow 0$  \vs{at a \sv{geometric} rate}.

\noindent (ii)
\vs{Let $R_0$ be} such that $a \vs{=\tfrac{1+\tfrac{1}{R_0}}{1+\delta}} < \eta$, and let $\eta = La$, where $L > 1$. \vs{From~\eqref{geometric decay}, for $k \geq 2$}:
\begin{align} \label{errorbound-geometric decay}
\notag
r_{k} & \leq a^k e^\alpha r_0 + C_{3,4} \left( \sum_{i=0}^{k-2} \eta^i a^{k-i-1} e^{\beta} + \eta^{k-1} \right) 
 = \left( \frac{\eta}{L} \right)^k e^\alpha r_0 + C_{3,4} \left( \sum_{i=0}^{k-2} \eta^i \left( \frac{\eta}{L} \right)^{k-i-1} e^{\beta} + \eta^{k-1} \right) \\
\notag
& \leq \left( \frac{\eta}{L} \right)^k e^\alpha r_0 + C_{3,4} e^\beta \eta^{k-1} \left( \sum_{i=0}^{k-2} \left( \frac{1}{L} \right)^{k-i-1} + 1 \right) 
 = \left( \frac{\eta}{L} \right)^k e^\alpha r_0 + C_{3,4} e^\beta \eta^{k-1} \left(\frac{1-\tfrac{1}{L^k}}{1-\tfrac{1}{L}}\right) \\
& \leq \eta^k \left( \frac{e^\alpha r_0}{L} + \frac{C_{3,4} e^\beta}{\eta(1 - 1/L)}  \right)  \triangleq R_k.
\end{align}
\yxf{Suppose that $\bar K \triangleq \min \{ K: r_k \leq \epsilon, \forall k \geq K \}$. Then $\bar K \leq \min \{ K : R_k \leq \epsilon, \forall k \geq K \} = \min \{ K: R_K \leq \epsilon, R_{K-1} > \epsilon \} = \left\lceil \log_{\eta} \left( \epsilon \Big/  \left( \frac{e^\alpha r_0}{L} + \frac{C_{3,4} e^\beta}{\eta(1 - 1/L)} \right) \right)  \right\rceil. $}\\ 
{Since the overall iteration (sample) complexity $N{(\epsilon)}$ is a summation of 
\vs{the iteration (sample) complexity at} each {major iteration}, we \vs{obtain} the following sequence
	of inequalities}:
\begin{align*}
 N{(\epsilon)} & = \sum_{k=1}^{\bar K} \left( (T_k^y - 1) + (T_k^x - 1) \right) 
 \leq \sum_{k=1}^{\bar K} ( \max\{ \lceil \tfrac{T}{\eta^{k-1}} \rceil ,K_x \}  + \max\{ \lceil \tfrac{T}{\eta^{k-1}} \rceil ,K_y \} - 2 ) \\ 
& \leq \sum_{k=1}^{\bar K} ( \tfrac{T}{\eta^{k-1}} + 1 + K_x  +  \tfrac{T}{\eta^{k-1}} +1 + K_y - 2 ) 
 = 2T\sum_{k=1}^{\bar K} (\tfrac{1}{\eta^{k-1}}) + {\bar K}(K_x + K_y)  \\
& = 2T \cdot \tfrac{\tfrac{1}{\eta^{\bar K}} - 1}{ \tfrac{1}{\eta} - 1} + {\bar K}(K_x + K_y) 
 \leq 2T \cdot \tfrac{\tfrac{1}{\eta^{\bar K}} }{ \tfrac{1}{\eta} - 1} + {\bar K}(K_x + K_y) 
 = \tfrac{2T \eta}{1-\eta} \cdot \tfrac{1}{\eta^{\bar K}} + {\bar K}(K_x + K_y) \\
& \leq \tfrac{2T \eta}{1-\eta} \cdot \tfrac{1}{\eta} \cdot \tfrac{1}{\epsilon} \left( \tfrac{e^\alpha r_0}{L} + \tfrac{C_{3,4} e^\beta}{\eta(1 -\tfrac{1}{L})}  \right) + \left( \log_{\eta}\left( \epsilon \big/ \left( \tfrac{e^\alpha r_0}{L} + \tfrac{C_{3,4} e^\beta}{\eta(1 - \tfrac{1}{L})}  \right) \right) + 1 \right)(K_x + K_y) \\
& = \tfrac{2T}{1-\eta} \left( \tfrac{e^\alpha r_0}{L} + \tfrac{C_{3,4} e^\beta}{\eta(1 -\tfrac{1}{L})}  \right) \tfrac{1}{\epsilon}  + \tfrac{K_x + K_y}{\log(1/\eta)} \log \left( \tfrac{1}{\epsilon} \right) 
  + 
 \left( 1 - \log_{\eta} \left( \tfrac{e^\alpha r_0}{L} + \tfrac{C_{3,4}
		 e^\beta}{\eta(1 - \tfrac{1}{L})}  \right) \right)(K_x + K_y) \\
& \leq  \tfrac{2T}{1-\eta} \left( \tfrac{e^\alpha r_0}{L} + \tfrac{C_{3,4} e^\beta}{\eta(1 - 1/L)}  \right) \tfrac{1}{\epsilon}  + \tfrac{K_x + K_y}{\log(1/\eta)}  \log \left( \tfrac{1}{\epsilon} \right) \\
& + \left( 1 - \log_{\eta} \left( \tfrac{e^\alpha r_0}{L} + \tfrac{C_{3,4} e^\beta}{\eta(1 - 1/L)}  \right) \right)(K_x +
			 K_y) \log \left( \tfrac{1}{\epsilon} \right) 
 = \left( \tfrac{2T}{1-\eta} \left( \tfrac{e^\alpha r_0}{L} + \tfrac{C_{3,4} e^\beta}{\eta(1 - \tfrac{1}{L})}  \right)  \right)  \tfrac{1}{\epsilon} \\
&  + \left( \left( 1 - \log_{\eta} \left( \tfrac{e^\alpha r_0}{L} + \tfrac{C_{3,4}
		 e^\beta}{\eta(1 - 1/L)}  \right) \right)(K_x +
			 K_y) + \tfrac{K_x + K_y}{\log(1/\eta)} \right) \log \left( \tfrac{1}{\epsilon} \right) 
 = \mathcal{O}\left(\tfrac{1}{\epsilon}\right).
\end{align*}
\end{proof}
\begin{remark}
Note that this complexity matches that of standard SA schemes, where to solve a smooth and strongly convex stochastic optimization problem to $\epsilon$-accuracy (i.e. $\bE[\| x_k - x^* \|^2] \leq \epsilon$) requires $\mathcal{O}(1/\epsilon)$ iterations (or equivalently sampled gradients). \qed
\end{remark}

\section{Numerical experiments}\label{sec: Numerics-SIADMM}
In this section, we compare SI-ADMM with \sv{two} other algorithms on {two test problems}: (i) LASSO with expectation-valued loss; and (ii) Distributed regression. 

\noindent \vs{A. \bf LASSO.} 
{Consider the following problem:} 
\begin{align} \label{LASSO - SIADMM} \tag{LASSO}
 \min \quad \bE[ ( l{(\xi)}^T x - s{(\xi)} )^2 ] + \bar \gamma \| x \|_1, 
\end{align}
where $x \in \Real^n, (l,s): \Xi \rightarrow \bR^n \times \bR$, and $(\Xi, \scrF, \bP)$ denotes the probability space. {We consider this problem first in order to compare SI-ADMM with existing stochastic variants of ADMM, which could only tackle problems with simple and closed-form $g$ function (regularizers).}

{In \vvs{subsequent discussions}, we denote $l(\xi),s(\xi)$ as $l,s$ for simplicity.} Then \eqref{LASSO - SIADMM} can be reformulated as an \eqref{SOpt} with one part of objective being stochastic and the other deterministic:
\begin{align} \label{strucLASSO - SIADMM}
\min & \ \bE[ ( l^T x - s )^2 ] + \bar \gamma \| y \|_1 \quad \st \ x - y = {\bf 0}.
\end{align}
{We compare the behavior of SI-ADMM} with stochastic
ADMM~\cite{ouyang2013stochastic}. {In particular, we rely on  synthetic} data:
$l = (\tilde{l};1)$, ${\tilde l} \sim N({\bf 0},
\Sigma_l), \Sigma_l = (\sigma_l^2 \times 0.5^{| i - j |} )_{i,j},$
$\vvs{s = l^Tx_{\rm true} + \epsilon_s}$, where $\epsilon_s \sim
N( 0, \sigma_s^2)$. {We assume that $l$} and $\epsilon_s$ are
independent \vvs{random variables}. To generate a deterministic $x_{\rm
\vvs{true}}$, we \vvs{generate} $n-1$ i.i.d random \vvs{samples} where  $rs_i
\sim U[-50,50], i = 1,\hdots,n-1$. and let \begin{align}\label{x-true - SIADMM}
[x_{\rm true}]_i := 
\begin{cases}
rs_i & \mbox{if } | rs_i | \leq 5, \\
0 & \mbox{if } | rs_i | > 5.
\end{cases}
\end{align}
Furthermore, let $[x_{\rm true}]_n \sim B(1,p)$ (\vvs{Bernoulli}). Then
approximately $\frac{n}{10}$ \sv{of the} elements are \sv{nonzero}. {In our experiments}, we
set $\sigma_l^2 = \sigma_s^2 = 5$ and {the stochastic part of the
objective has an explicit deterministic form given by}  $(x - x_{\rm true})^T
\Sigma (x - x_{\rm true}) +
\sigma_s^2 $, as seen next. 
\begin{align*}
& \quad \ \bE[ ( l^Tx - s )^2 ]
 = \bE[ ( l^T x - (l^T x_{\rm true} + \epsilon_s ) )^2 ] 
 = \bE[ ( l^T(x - x_{\rm true}) - \epsilon_s )^2 ] \\
& = \bE[ (x - x_{\rm true})^T l l^T (x - x_{\rm true}) - 2 \epsilon_s l^T (x - x_{\rm true}) + \epsilon_s^2 ] \\
& = \bE[ (x - x_{\rm true})^T l l^T (x - x_{\rm true}) ] - 2 \bE[\epsilon_s] \bE[ l^T (x - x_{\rm true}) ] + \bE[\epsilon_s^2]\\
& = (x - x_{\rm true})^T \bE[ l l^T ] (x - x_{\rm true}) + \sigma_s^2 
 = (x - x_{\rm true})^T \Sigma (x - x_{\rm true}) + \sigma_s^2,
\end{align*}
where $\Sigma \triangleq \bE[ll^T] = {\tiny \pmat{ \Sigma_l & \\ &1}}$. Therefore, the optimal solution $x^*$ and {the optimal objective} $F^*$ can be obtained
by a deterministic proximal gradient algorithm \cite{parikh2014proximal}. 

\noindent {\em Implementation of \texttt{SADM0}~\cite{ouyang2013stochastic}:} The stochastic ADMM \vs{scheme (\texttt{SADM0})}  {utilizes the following steps:}
\begin{align}
\tag{$x-{\rm update}$ }	x_{k+1} &:= \displaystyle
\hspace{-0.025in} \frac{-2(l_k^Tx_k - s_k)l_k + \lambda_k + \rho y_k + \eta_k \cdot x_k}{\rho + \eta_k } \\
\tag{$y-{\rm update}$}	y_{k+1} &: = \displaystyle \hspace{-0.025in} S_{\frac{\bar \gamma}{\rho}} (x_{k+1} - \lambda_k / \rho) \\
\tag{$\lambda-{\rm update}$}	\lambda_{k+1} &:= \lambda_k - \rho( x_{k+1} - y_{k+1} ),
\end{align}
where $(l_k, s_k)$ {denote the samples generated} at iteration $k+1$,
$\eta_k = 1000 \sqrt{k}$, $x_0 = {\bf 0}$, $y_0 = {\bf 0}$, $\lambda_0 = {\bf 0}$, 
  and the soft-thresholding operator {$S_{\alpha}(\cdot)$} is defined as \vs{follows}.
\begin{align} \label{soft-thresh operator}
S_{\alpha} (x)\triangleq 
\begin{cases}
x_{i} - \alpha & \mbox{if } x_i > \alpha, \\
0 & \mbox{if }  -\alpha \leq x_i \leq \alpha, \\
x_{i} + \alpha & \mbox{if } x_i < -\alpha.
\end{cases}
\end{align}
 \noindent {\em Implementation of SI-ADMM:} The simplified version contains following updates in each iteration:
{
\begin{align} \tag{exact $y$-update}
& y_{k+1} = S_{\frac{\bar \gamma}{\rho}} (x_k - \lambda_k/\rho) \\
\notag
& x_{k,j+1} := x_{k,j} - \frac{\gamma_{x}}{j} [\nabla_x \tilde{\mathcal{L}}_{1}(x_{k,j},y_{k+1},\lambda_k,\xi^x_{k,j}) + P(x_{k,j} - x_k) ], 
\notag
 \qquad \quad j = 1, \hdots,T_{k+1}^x -1,  x_{k,1} := x_k,\\
& {x_{k+1} := x_{k, T_{k+1}^x}} \tag{Inexact $x$-update}	\\
& \tag{$\lambda$-update}	\lambda_{k+1} := \lambda_k - \rho( x_{k+1} - y_{k+1}),
\end{align}
where $\tilde{\Lscr}_1$ {is} defined in \eqref{L1-SIADMM}}. We stop the algorithm when the outer loop iteration number reaches maximum. \vs{The parameters are defined \yxf{and calculated} based on Lemma~\ref{Lm: errorbound-SIADMM}, Corollary~\ref{choice of delta - SI-ADMM} and the property of augmented Lagrangian function $\Lscr_\rho(x,y,\lambda)$}: $P = {\bf 0}, x_0 = {\bf 0}, \lambda_0 = {\bf 0}, $ 
$T_{k+1}^x = \max\{ K_x, \lceil T/\eta^k \rceil \}$, $T = 1000$, $\eta = {(1 + \delta/2)}/({1 + \delta})$, $K_x = \lceil \gamma_x^2 M_x / (2 c_x \gamma_x - 1) \rceil + 1$, $\gamma_x = 1/c_x, c_x = 2\lambda_{\min}(\Sigma) + \rho$, $M_x = L_x^2 + (1 + 1/R) v_1^x$, 
$R = 1$, $L_x = 2 \lambda_{\max}(\Sigma) + \rho$, $v_1^x = 8 \lambda_{\max}( \bE[(ll^T - \Sigma)^2 ])$,
$
\delta = 2 \left( \frac{\rho}{2 \lambda_{\min}(\Sigma)} + \frac{2 \lambda_{\rm max}(\Sigma)}{\rho} \right)^{-1}.
$
We derive $v^x_1$ and $v^x_2$ by first deriving a bound on $\bE [\|w^x\|^2]$. If $\Sigma_l = (\sigma_{ij})_{ij}, i,j \in \{ 1,\hdots,n-1 \}$, then from the definition of $w^x$ in Corollary~\ref{Corr: theobound},
\begin{align*}
& \quad \bE [ \| w^x \|^2 ] 
 =  \bE[ \| 2 l(l^Tx - s) - 2 \Sigma (x - x_{\rm true}) \|^2 ] 
 = \bE[ \| 2 l(l^Tx - (l^Tx_{\rm true} + \epsilon_s ) ) - 2 \Sigma (x - x_{\rm true}) \|^2 ] \\
& = \bE[ \| 2 ( l l^T - \Sigma )( x - x_{\rm true} ) - 2 l \epsilon_s \|^2 ] 
 = 4 ( x - x_{\rm true} )^T \bE[ (ll^T - \Sigma)^2 ] ( x - x_{\rm true} ) \\
&  - 8 \bE[\epsilon_s]\bE[ l^T(ll^T - \Sigma)(x- x_{\rm true}) ] + 4 \bE[l^Tl] \bE[\epsilon_s^2] \\
& \leq 8 x^T \bE[ (ll^T - \Sigma)^2 ] x + 8 x_{\rm true}^T \bE[ (ll^T - \Sigma)^2 ] x_{\rm true} + 4 \sigma_s^2 \sum_{i=1}^n \bE[l_i^2] \\
& = 8 x^T \bE[ (ll^T - \Sigma)^2 ] x + 8 x_{\rm true}^T \bE[ (ll^T - \Sigma)^2 ] x_{\rm true} + 4 \sigma_s^2 \left( \sum_{i=1}^{n-1} \sigma_{ii} + 1 \right)
\end{align*} 
Therefore, $v^x_1$ and $v^x_2$ {may be derived as follows}: 
\begin{align}
\label{v^x_2 - SIADMM}
& v^x_1  =  8 \lambda_{\max} \left( { \bE[ (ll^T - \Sigma)^2 ] } \right), \mbox{ and }
 v^x_2  = 8 x_{\rm true}^T \bE[ (ll^T - \Sigma)^2 ] x_{\rm true} + 4 \sigma_s^2 \left( \sum_{i=1}^{n-1} \sigma_{ii} + 1 \right). 
\end{align}
The following lemma shows we may compute $\bE[ (ll^T - \Sigma)^2 ]$ in closed form and we use the $\texttt{eig}$ function in $\texttt{Matlab}$ is to compute its eigenvalues.
\begin{lemma}
Suppose $l = (\bar{l};1), \bar{l} \sim N({\bf 0}, \Sigma_l), \Sigma_l \triangleq (\sigma_{ij})_{ij}, \Sigma \triangleq \bE[ll^T]$. If we let $V \triangleq (v_{ij})_{ij} \triangleq \bE[ (ll^T - \Sigma)^2 ] $, then 
\begin{align*}
v_{ij} = &
\begin{cases}
\sigma_{ij} \sum_{k=1}^{n-1} \sigma_{kk} + \sum_{k=1}^{n-1} \sigma_{ik} \sigma_{jk} + \sigma_{ij}, \quad \mbox{ if } i \neq n, j \neq n, \\
0, \quad \mbox{ if } i = n, j \neq n,
 \mbox{or } i \neq n, j = n, \\
\sum_{k=1}^{n-1} \sigma_{kk},  \quad \mbox{ if } i = j = n.
\end{cases}
\end{align*}
\end{lemma}
\begin{proof}
	The proof \vs{is omitted and utilizes} Isserlis Theorem~\cite{10.2307/2331932} to calculate 4th-order moments. 
\end{proof}
\noindent {\em \vvs{Implementation of  \texttt{SADM1}~\cite{azadi2014towards}}:} In {recent work}, {Azadi and Sra~\cite{azadi2014towards} prove a faster rate of convergence of ${\cal O}(1/\sv{K})$ in
terms of sub-optimality and infeasibility (\sv{where $K$ denotes the number of iterations}) for the stochastic ADMM scheme applied to  strongly convex risk functions 
by utilizing an alternate form of averaging}.
{In applying \sv{the resulting scheme (referred to as \texttt{SADM1})} to~\eqref{strucLASSO - SIADMM},
the following steps are taken \sv{at} iteration $k+1$.} \begin{align}
\tag{$x$ - update}
x_{k+1} & = \frac{-2 (l_k^T x_k - s_k) l_k + \lambda_k + \rho y_k + \sv{\tfrac{(k+2)\mu_f x_k}{2}}}{\rho + \sv{\tfrac{(k+2)\mu_f}{2}}}, \\
\tag{$y$ - update}
y_{k+1} & = S_{\frac{\bar \gamma}{\rho}} (x_{k+1} - \lambda_k / \rho), \\
\tag{$\lambda$ - update}
\lambda_{k+1}  & = \lambda_k - \rho (x_{k+1} - y_{k+1}),\\ 
\tag{$x$ - average}
\bar{x}_{k+1} & = \frac{2}{(k+2)(k+3)} \sum_{j=0}^{k+1}(j+1)x_j. \\
\tag{$y$ - average}
\bar{y}_{k+1} & = \frac{2}{(k+1)(k+2)} \sum_{j=1}^{k+1} j y_j,
\end{align}
where $\mu_f = 2 \lambda_{\min}(\Sigma)$ denotes the convexity constant of $\bE[ ( l^T x - s )^2 ]$. 

\noindent {{\bf (A1)} \em Insights from \sv{comparisons} with \texttt{SADM0}:} {We compare stochastic ADMM and
SI-ADMM in Table~\ref{tab: SIADMM vs. SADM0}, where each scheme uses the same samples by
fixing the seed in the random number generator and $x_1^*$, $x_2^*$ denote the
solutions given by SI-ADMM and \texttt{SADM0}, respectively.}
\begin{table}[htbp]
{
\begin{center}
\begin{tabular}{| c | c |c| c | c |c| c |}
\hline
dim. & $\rho$ & Sample & Tm$_1$(s) & Tm$_2$(s) & $\|x_1^*-x^*\|^2$  & $ \|x_2^*-x^*\|^2 $  \\ 
\hline
10 & 100 & 10919 & 0.39 & 1.04 & 5.82e-01 & 3.79e-01  \\ 
\hline 
10 & 100 & 70995 & 2.58 & 6.87 & 1.41e-01 & 9.99e-02  \\ 
\hline 
10 & 100 & 189144 & 6.75 & 17.91 & 3.15e-02 & 2.29e-02  \\ 
\hline 
10 & 50 & 11873 & 0.43 & 1.12 & 3.89e-01 & 3.67e-01  \\ 
\hline 
10 & 50 & 113692 & 4.02 & 10.67 & 2.21e-02 & 5.52e-02  \\ 
\hline 
10 & 50 & 485538 & 17.37 & 44.88 & \blue{5.44e-04} & \blue{1.71e-03}  \\ 
\hline 
10 & 20 & 14563 & 0.53 & 1.30 & 1.18e-01 & 3.32e-01  \\ 
\hline 
10 & 20 & 400799 & 14.30 & 37.12 & 7.13e-05 & 3.38e-03  \\ 
\hline 
10 & 20 & 6336323 & 224.62 & 631.96 & \blue{1.00e-04} & \blue{1.13e-04}  \\ 
\hline 
100 & 100 & 10919 & 1.67 & 2.65 & 1.47e+01 & 5.72e+00  \\ 
\hline 
100 & 100 & 70970 & 10.73 & 17.55 & 9.55e-01 & 4.91e-01  \\ 
\hline 
100 & 100 & 189018 & 28.48 & 50.10 & 8.82e-02 & 5.09e-02  \\ 
\hline 
100 & 50 & 11867 & 1.78 & 3.26 & 6.63e+00 & 5.33e+00  \\ 
\hline 
100 & 50 & 113377 & 17.20 & 28.06 & 5.45e-02 & 1.87e-01  \\ 
\hline 
100 & 50 & 482823 & 72.51 & 124.12 & \blue{1.26e-03} & \blue{2.62e-03}  \\ 
\hline 
100 & 20 & 14484 & 2.20 & 3.85 & 1.13e+00 & 4.39e+00  \\ 
\hline 
100 & 20 & 386699 & 58.10 & 99.65 & 1.00e-03 & 6.11e-03  \\ 
\hline 
100 & 20 & 5893508 & 884.55 & 1529.41 & \blue{1.39e-04} & \blue{1.76e-04}  \\ 
\hline 
\end{tabular}
\end{center}
\caption{(1)SI-ADMM vs.(2) \texttt{SADM0} on LASSO ($\bar \gamma = 0.1$)} 
\label{tab: SIADMM vs. SADM0}
}
\vspace{-0.3in}
\end{table}
\begin{figure}[htbp]
\vspace{-0.2in}
\begin{center}
\includegraphics[width=3in]{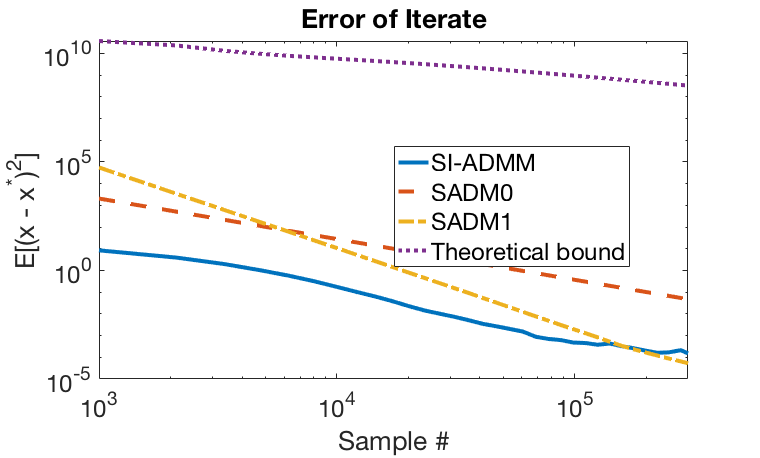}
\includegraphics[width=3.2in]{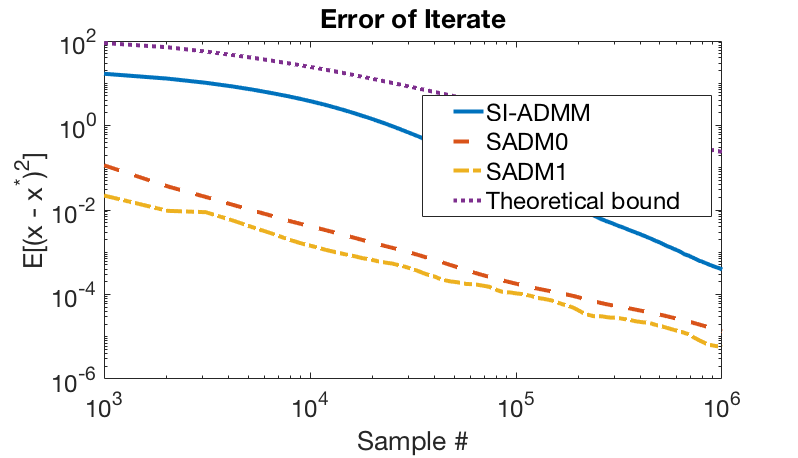}
\end{center}
\caption{SI-ADMM vs. \texttt{SADM0} vs. \texttt{SADM1}, 
$\rho = 10, n = 10$ (L) and 
$\rho = 50, n = 10$ (R), \yxf{log-scale}.}
\label{Fig: SIADMMvsSADM1}
\end{figure}

As seen \vs{in} Table~\ref{tab: SIADMM vs. SADM0}, for the same sample size and dimension, the
performances of SI-ADMM and \texttt{SADM0} vary \vs{with the choice of }
$\rho$. {When $\rho$ is small, SI-ADMM may outperform \texttt{SADM0} given
enough data (highlighted in blue).} Meanwhile, SI-ADMM {generally requires}
less time than \texttt{SADM0} {for the same \vs{sample complexity}}. {After all,
\texttt{SADM0} needs to compute an $x$-update, a $y$-update, and a $\lambda$-update
\vs{for every sampled gradient}, while SI-ADMM just
needs a gradient step. }

\noindent {{\bf (A2)} \em Insights from comparing with \texttt{SADM0} and
\texttt{SADM1}.} To compare SI-ADMM with \texttt{SADM0} and \texttt{SADM1} on 
\eqref{strucLASSO - SIADMM}, we plot trajectories of mean error of iterate $x_k$, shown in Figure~\ref{Fig: SIADMMvsSADM1}. {Under the same experimental settings, we add an averaging step to \texttt{SADM0}, defined as follows in the $(k+1)$th
iteration: $\bar{x}_{k+1} \vvs{ \triangleq } \frac{1}{k+2} \sum_{i = 0}^{k+1}
x_i$ and $\bar{y}_{k+1} \vvs{ \triangleq } \frac{1}{k+1} \sum_{i = 1}^{k+1}
y_i$. 
Furthermore, let $\eta_k = \mu k$ in \texttt{SADM0}, where $\mu_f = 2 \lambda_{\min} (\Sigma)$ \vvs{denotes the
strong convexity} constant. During the
implementation of \texttt{SADM1}, we take $x_0 = {\bf 0}, y_0 = {\bf 0},$ and $\lambda_0 = {\bf 0}$ as initial
values for iterates}. We generate 10 replications for each algorithm starting
from the same initial point to calculate the empirical error of iterate. Specifically, we
calculate empirical mean error of iterate as follows: for
SI-ADMM at iteration $k$, suppose the 10 realizations of iterate $x_k$ are
$x_{k,1}, \hdots, x_{k,10}$. Then we use 
$\frac{1}{10} \sum_{i=1}^{10} \|
x_{k_i} - x^* \|^2$ to approximate $\bE[\| x_k - x^* \|^2]$. Likewise, we
estimate empirical mean error of iterate at $\bar{x}_k$
for {\texttt{SADM0} and \texttt{SADM1}}. The \yxf{left} plot in Figure~\ref{Fig: SIADMMvsSADM1} {displays} the
trajectories of empirical mean error of iterate when $n = 10, \rho = 10$, while \yxf{the right} plot in Figure~\ref{Fig: SIADMMvsSADM1} is
generated when $\rho$ is {raised} to $50$. {The} x-axis {denotes}  the number of samples
$(l_k,s_k)$ drawn from the pre-specified distributions. The theoretical upper bound of
$\bE[ \| x_k - x^* \|^2]$ of SI-ADMM is calculated based on
Corollary \ref{Corr: theobound}. In particular, we take \vs{expectations} on both
sides of \eqref{theobound-supermartingale}, replace $T_{k+1}^x$ with
$T/\eta^k$, and obtain the following 
inequality for any $k \geq 0$ and $R_{0,k} > 0$: \vs{\begin{align}
\label{Theobound-SIADMM}
 \bE[ \| u_{k+1} - u^* \|_G^2] \leq \left( \frac{(1 + R_{0,k})\zeta_2}{{T}/{\eta^k}} + \frac{1+\tfrac{1}{R_{0,k}}}{1+\delta} \right) \bE[ \|u_k - u^*\|_G^2 ] + \frac{ (1+R_{0,k}) \zeta_1}{{T}/{\eta^k}},  
\end{align}}
where \vs{$\zeta_1, \zeta_2$} 
are defined as follows, \yxf{(see proof of Corollary~\ref{Corr: theobound} in 
Appendix)
\begin{align}
\notag
& \zeta_2 = (\lambda_{\max}(\hat{P}) + \rho\gamma \| A \|^2 ) C_{1,2}^x, \quad \zeta_1 = ( \lambda_{\max}(\hat{P}) + \rho \gamma\| A \|^2 ) C_{1,1}^x, \\
\notag
& C_{1,2}^x = \left( \frac{\gamma_x^2}{2c_x\gamma_x - 1 - \gamma_x^2 M_x / K_x } + K_x b_x \right)  \frac{2\sv{v_1^x (1 + R)}}{\lambda_{\min} (\hat{P})(1+\delta) } 
 + \frac{2\sv{K_x a_x }}{\lambda_{\min}(\hat{P})} \left(1+ \frac{1}{1+\delta} \right), \\
\notag
& C_{1,1}^x = \left( \frac{\gamma_x^2}{2c_x\gamma_x - 1 - \gamma_x^2 M_x / K_x } + K_x b_x \right)\left( \sv{2}v_1^x(1+R) \| x^* \|^2 + v_2^x \right).
\end{align}
}
Recall that $u_k = (x_k;y_k;\lambda_k), A = I_{n \times n}, \hat{P} = P + \rho A^T A, v_2^x =  8 x_{\rm true}^T \bE[ (ll^T - \Sigma)^2 ] x_{\rm true} + 4 \sigma_s^2 ( \sum_{i=1}^{n-1} \sigma_{ii} + 1 )$ according to \eqref{v^x_2 - SIADMM}, and
\begin{align*}
& a_x =  \prod_{i=1}^{K_x - 1} a_i, \quad  a_i = \yux{1 - \tfrac{2c_x \gamma_x}{i} + \tfrac{M_x \gamma_x^2}{i^2} }, 
 \quad b_x = \sum_{i=1}^{K_x-2} ( \tfrac{\gamma_x^2}{i^2}  a_{i+1} \cdot \hdots \cdot a_{K_x-1}) + \sv{\tfrac{\gamma_x^2}{(K_x-1)^2}},  
\end{align*}
based on Remark~\ref{Choices of a and b}. Last, we set $R = 1$. \vvs{The values of the} other parameters are already specified during the implementation of SI-ADMM. Noting that $\bE[ \| x_{k+1} - x^* \|^2 ] \leq \frac{1}{\rho} \bE[ \| u_{k+1} - u^* \|_G^2 ]$, the upper bound for $\bE[ \| x_{k+1} - x^* \|^2 ]$ is calculated through dividing the right-hand side of \eqref{Theobound-SIADMM} by $\rho$. Here $\bE [\| u_k - u^* \|_G ^2 ]$ is approximated by sample average calculated by the algorithm and
\yxf{ 
\begin{align*}
& R_{0,k} \triangleq \sv{\sqrt{ \tfrac{\bE [\| u_k - u^* \|_G ^2 ]}{(1+\delta) \mathcal{D}} }},\quad \mathcal{D} = \sv{\tfrac{\zeta_2 \eta^k\bE [\| u_k - u^* \|_G ^2 ]}{T}   + \tfrac{ \zeta_1 \eta^k}{T}}.
\end{align*}}
From Figure~\ref{Fig: SIADMMvsSADM1}, we {may} draw similar conclusions. When $\rho$ is {smaller}, {the} performance of SI-ADMM is {generally} better. As seen {from} Figure~\ref{Fig: SIADMMvsSADM1}(\yxf{Left}), at the outset, SI-ADMM converges faster than the other two and is comparable to {SADM1} when {the} sample size is large.

\noindent \vs{\bf B. Distributed regression.} 
Next, we aim to solve \eqref{SOpt} \vvs{in a distributed fashion}. 
While one agent {may} obtain
{noise-corrupted sampled} gradients of $\tilde{f}$, the other has access to
{noise-corrupted sampled} gradients of $\tilde{g}$.

\noindent {\em Problem description:} Consider \eqref{SOpt}. Specifically, the experiment is designed as follows:
$f(x) = \bE[ ( l_1{(\xi)}^T x - s_1{(\xi)} )^2 ] $, $g(y) = \bE[ ( l_2\vvs{(\xi)}^T y - s_2\vvs{(\xi)} )^2 ]$. Denote $l_1(\xi),l_2(\xi),s_1(\xi),s_2(\xi)$ as $l_1,l_2,s_1,s_2$. Let $l_1 \sim N({\bf 0}, \Sigma_1), l_2 \sim N({\bf 0}, \Sigma_2), \Sigma_1 = \Sigma_2 \triangleq \Sigma =  (\sigma_l^2 \times 0.5^{|i-j|} )_{ij}, s_1 = l_1^T \beta_1 + \epsilon_{s1}, s_2 = l_2^T \beta_2 + \epsilon_{s2},  \epsilon_{s1} \sim N({\bf 0} , \sigma_s^2),  \epsilon_{s2} \sim N({\bf 0} , \sigma_s^2)$. Denote $\Sigma = (\sigma_{ij})_{ij}$, and
$A = (a_{ij})_{ij}$. Let $$a_{ij} = \begin{cases}
0, & \mbox{ if } i > j, \\
3, & \mbox{ if } i = j, \\
\sim N(0,0.01), & \mbox{ if } i < j.
\end{cases}$$
$B = -I_n, b = {\bf 0}, \sigma_l^2 = \sigma_s^2 = 5.$ Let $rs_i \sim U[-50,50], i = 1,\hdots,n$ \vvs{then $\beta_2$ is defined as follows:}
\begin{align}\label{beta-SIADMM}
[\beta_2]_i \triangleq 
\begin{cases}
rs_i & \mbox{if } | rs_i | \leq 5, \\
0 & \mbox{if } | rs_i | > 5.
\end{cases}
\end{align}
Then let $\beta_1 = A^{-1} \beta_2$. The optimal solution $z^* \triangleq (x^*; y^*) =
(\beta_1; \beta_2)$ and objective function value $F^*$ equals to $2\sigma_s^2$
{since} $f$ and $g$ are given by $f(x) = (x - \beta_1)^T
\Sigma_1 (x - \beta_1) + \sigma_s^2 $, $g(y) = (y - \beta_2)^T \Sigma_2 (y -
\beta_2) + \sigma_s^2 $ and $(\beta_1;\beta_2)$ satisfies the constraint $Ax -
y = {\bf 0}$. {These closed-form expressions may be derived \sv{in a  similar fashion} to  that} in
Section~\ref{sec: Numerics-SIADMM} (A). This experiment fits well in the setting where two
agents want to estimate the true value $\beta_1$ and $\beta_2$, respectively,
given appropriate sample inputs and outputs. Furthermore, they know $\beta_1$
and $\beta_2$ {are related through}  $\beta_2 = A \beta_1$, \yxb{and this
relation can be enforced through the constraint: $Ax-y= {\bf 0}$.} Moreover,
\sv{agents} \sv{utilize local computations that do not require sharing data.} 
We compare SI-ADMM with a distributed stochastic approximation (DSA) algorithm \sv{for this problem}.

\noindent {\em Implementation of SI-ADMM:} SI-ADMM contains the following updates in iteration $k+1$:
{
\begin{align} 
\notag
 y_{k,j+1} &:= y_{k,j} - \frac{\gamma_{y}}{j}\left(\nabla_y  \tilde{\mathcal{L}}_{2}(x_k,y_{k,j},\lambda_k,\xi^y_{k,j}) + Q(y_{k,j} - y_k)\right) , 
 j = 1, \hdots, T_{k+1}^y -1, y_{k,1} := y_k ,\\
 y_{k+1} & := y_{k, T_{k+1}^y} \tag{Inexact $y$-update} \\
\notag
 x_{k,j+1} &:= x_{k,j} - \frac{\gamma_{x}}{j} \left(\nabla_x \tilde{\mathcal{L}}_{1}(x_{k,j},y_{k+1},\lambda_k,\xi^x_{k,j}) + P(x_{k,j} - x_k)\right), 
 j = 1, \hdots,T_{k+1}^x -1,  x_{k,1} := x_k,\\
 x_{k+1} & := x_{k, T_{k+1}^x} \tag{Inexact $x$-update} \\
 \tag{$\lambda$-update}	\lambda_{k+1} &:= \lambda_k - \rho( Ax_{k+1} - y_{k+1}), 
\end{align}
where $\tilde{\Lscr}_1$ and $\tilde{\Lscr}_2$ are defined in \eqref{L1-SIADMM} and \eqref{L2-SIADMM}.} Values for parameters are assigned or computed as follows, based on Corollary~\ref{choice of delta - SI-ADMM}, Lemma~\ref{Lm: errorbound-SIADMM}, and property of augmented Lagrangian $\Lscr_\rho(x,y,\lambda)$:
\begin{align*}
& Q = \rho I_{n \times n}, P = {\bf 0}, R = 1, T = 1000, x_0 = {\bf 0}, 
 y_0 = {\bf 0}, \lambda_0 = {\bf 0}, \eta = \frac{1}{1 + \delta}, \\
&  T_{k+1}^y = \max\{ K_y, \lceil T/\eta^k \rceil \}, T_{k+1}^x = \max\{ K_x, \lceil T/\eta^k \rceil \}, \\
& \delta = \min \left\{ \frac{4\lambda_{\min}(\Sigma)}{\rho}, 2 \left( \frac{\rho \lambda_{\max}(A^TA) }{2 \lambda_{\min}(\Sigma)} + \frac{2 \lambda_{\rm max}(\Sigma)}{\rho \lambda_{\min}(AA^T) } \right)^{-1} \right\}\\
& K_x = \lceil \gamma_x^2 M_x \rceil + 1, K_y = \lceil \gamma_y^2 M_y \rceil + 1, 
 \gamma_x = 1/c_x, c_x = 2\lambda_{\min}(\Sigma) + \rho, M_x = L_x^2 + 2 v_1^x, \\
& L_x = 2 \lambda_{\max}(\Sigma) + \rho,  \gamma_y = 1/c_y, c_y = 2\lambda_{\min}(\Sigma) + 2\rho, M_y = L_y^2 + 2 v_1^y,
 L_y = 2 \lambda_{\max}(\Sigma) + 2\rho, \\
& v_1^x =  v_1^y = 8 \lambda_{\max}( \bE[(l_1l_1^T - \Sigma)^2 ] ) = 8 \lambda_{\max}( \bE[(l_2l_2^T - \Sigma)^2 ] ).
\end{align*}
We run the outer loop for $100$ iterations before \sv{termination} and \yux{set $\rho = 20$ while $v_1^x$ and $v_1^y$ are derived as a fashion similar to that in Section~\ref{sec: Numerics-SIADMM} (A), as captured by the next Lemma.}
\begin{lemma}\label{closed form V 2-SIADMM}
Suppose $l \sim N({\bf 0}, \Sigma)$. If $V \triangleq (v_{ij})_{ij} \triangleq \bE[ (ll^T - \Sigma) ] $, then
$v_{ij} = \sigma_{ij} \sum_{k=1}^n \sigma_{kk} + \sum_{k=1}^n \sigma_{ik} \sigma_{jk}.$
\end{lemma}
\noindent {\em Theoretical bounds:} The theoretical bounds for $\bE[ \| u_{k+1} - u^* \|_G^2 ]$ are developed based on Lemma~\ref{lm: main recursive inequality}(i), \yxf{where we obtain the following by using \eqref{supermartigale inequality 0} and $T_{k+1}^y \geq T/\eta^k $, $T_{k+1}^x \geq T/\eta^k $:
\begin{align*}
\notag
\bE[\| u_{k+1} - u^* \|^2_G ] & \leq \left( (1+R_{0,k}) \left( \frac{\zeta_2^x + \zeta_2^y}{T / \eta^k} + \frac{\zeta_2^{xy} }{T^2 / \eta^{2k}} \right) + \frac{1 + 1/R_{0,k}}{1 + \delta} \right)  \vs{\mathbb{E}[\| u_k - u^* \|_G^2]} \ \\
\notag
& + (1+R_{0,k}) \left( \frac{ \zeta_1^x + \zeta_1^y }{T / \eta^k} + \frac{\zeta_1^{xy}}{T^2/ \eta^{2k} } \right).
\end{align*}
}
\yxf{Formulas of $\zeta_1^x,\zeta_1^y,\zeta_1^{xy}, \zeta_2^x,\zeta_2^y,\zeta_2^{xy}$ \vs{may be found} in the proof of Lemma~\ref{lm: main recursive inequality}(i). To calculate these constants,} \sv{some} additional parameters including $a_x, b_x, a_y, b_y, v_2^x, v_2^y$ need to be known besides \sv{those} already computed during the implementation of SI-ADMM. In particular, as indicated by Remark~\ref{Choices of a and b},
\begin{align*}
& a_i^x = \left(1 - \tfrac{2c_x \gamma_x}{i} + \tfrac{M_x \gamma_x^2}{i^2}\right), 
 a_i^y = \left(1 - \tfrac{2c_y \gamma_y}{i} + \tfrac{M_y \gamma_y^2}{i^2} \right), 
 a_x =  \prod_{i=1}^{K_x - 1} a_i^x, a_y =  \prod_{i=1}^{K_y - 1} a_i^y, \\
 & b_x = \sum_{i=1}^{K_x-2} \left( (\gamma_x/i)^2  a_{i+1}^x \cdot \hdots \cdot a_{K_x-1}^x \right) + \tfrac{ \gamma_x^2}{(K_x-1)^2}, 
  b_y = \sum_{i=1}^{K_y-2} \left( \tfrac{\gamma_y^2}{i^2}  a_{i+1}^y \cdot \hdots \cdot a_{K_y-1}^y \right) + \tfrac{ \gamma_y^2}{(K_y-1)^2}.
\end{align*}
Meanwhile, $v_2^x = \beta_1^T \bE[ (l_1l_1^T - \Sigma)^2 ] \beta_1 + 4 \sigma_s^2(\sum_{k=1}^n \sigma_{kk})$, $v_2^y = \beta_2^T \bE[ (l_2 l_2^T - \Sigma)^2 ] \beta_2 + 4 \sigma_s^2(\sum_{k=1}^n \sigma_{kk})$. These formulas are derived in the same fashion as in  Section~\ref{sec: Numerics-SIADMM} (A). Computation of $\bE[ (l_1l_1^T - \Sigma)^2 ]$ and $\bE[ (l_2l_2^T - \Sigma)^2 ] $ can be referred to Lemma~\ref{closed form V 2-SIADMM}.  For all computations that involve $R$, set $R = 1$. $\bE[ \| u_k - u^* \|_G^2 ]$ is estimated by sample average obtained through running the algorithm 10 times.  At last, we set $R_{0,k} = \sqrt{ \tfrac{\bE[ \| u_k - u^* \|_G^2 ]}{(1+\delta)\mathcal{D}} }$, where
\yxf{
\begin{align*}
 \mathcal{D} = \left( \frac{ \zeta_2^x + \zeta_2^y +  \frac{ \zeta_2^{xy} }{T} \eta^k}{T} \right) \eta^k \bE[\| u_k - u^* \|_G^2] 
 + \left( \frac{ \zeta_1^x + \zeta_1^y +  \frac{ \zeta_1^{xy} }{T} \eta^k}{T} \right) \eta^k,
\end{align*}
}
and $T, \eta$ \sv{assume the same values as}  in SI-ADMM implementation.
Notice that $\bE[ \| x_k - x^* \|^2 + \| y_k - y^* \|^2 ] \leq \bE[ \| u_k - u^* \|_G^2 ] / \min\{ \lambda_{\min}(Q) , \lambda_{\min}(\hat{P}) \} $. Therefore, we obtain upper bound for $\bE[ \| x_k - x^* \|^2 + \| y_k - y^* \|^2 ]$ by dividing the bound for $\bE[ \| u_k - u^* \|_G^2 ] $ by $\min\{ \lambda_{\min}(Q) , \lambda_{\min}(\hat{P}) \}$.

\noindent {\em Implementation of DSA:} {The distributed SA (DSA) is a simple modification of standard SA {where the} gradient step for either $x$ or $y$ can be computed locally and {we employ a  projection onto a compact ball containing the solution to improve performance. }
In particular, (DSA)  contains the following steps in iteration $k$:}
\begin{align}
\tag{$x$ - update}
 \tilde{x}_{k+1}  & = x_k - \frac{1}{\mu_f k} \left( 2(l_{k,x}^Tx_k - s_{k,x})l_{k,x} \right), \\
\tag{$y$ - update}
 \tilde{y}_{k+1}  & = y_k - \frac{1}{\sigma_g k} \left( 2(l_{k,y}^Ty_k - s_{k,y})l_{k,y} \right), \\
\tag{$x$-projection}
 \bar{x}_{k+1} & = (I_n + A^TA)^{-1}( \tilde{x}_{k+1} + A^T \tilde{y}_{k+1} ), \\
\tag{$y$-projection}
 \bar{y}_{k+1}  &= A \bar{x}_{k+1}, \\
\tag{final projection}
 (x_{k+1};y_{k+1}) & = P_{Z} (\bar{x}_{k+1}; \bar{y}_{k+1})
\end{align}
where $(l_{k,x}, s_{k,x})$ and $(l_{k,y}, s_{k,y})$ are samples of $(l_1,s_1)$
and $(l_2,s_2)$ drawn at iteration $k$, respectively. $P_Z$ denotes the projection onto set $Z$. Note that $x$-update and
$y$-update can be computed by each agent and the projections could be finished
via a center or by each agent if they exchange $\tilde{x}_{k+1},
\tilde{y}_{k+1}$. During the implementation, we take $\mu_f = \sigma_g = 2
\lambda_{\min}(\Sigma)$ as the strongly convex constants, $x_0 = {\bf 0}, y_0 =
{\bf 0}$ as initial points. The final step requires projection onto a compact
ball $Z \triangleq \{ z: \| z \|^2 \leq  \Gamma \| (x^*;y^*) \|^2 \}$. The
choice of $\Gamma$ is vital to the performance of the distributed SA. In fact,
it reflects how confident one is about the location of $(x^*;y^*)$. 

\noindent {\bf Comparison between SI-ADMM and DSA:} Trajectories of error of iterates with respect to sample batch number are shown in the following figures. Note that a batch of samples contains two pairs of data points: $(l_{k,x}, s_{k,x}), (l_{k,y}, s_{k,y})$. During the implementation, 10 trajectories are generated for each algorithm, and we approximate the expectation $\bE[ \| x_k - x^* \|^2 + \| y_k - y^* \|^2 ]$ by sample average.  Figure~\ref{SIADMMvsDSA1} compares the performance of DSA against SI-ADMM. $\rho = 20$ and $\Gamma$ is chosen as $50, 5000$, and $500000$. $z_k  = (x_k; y_k), z^* = (x^*; y^*)$. Performance of DSA improves when $\Gamma$ decreases, implying estimation of $(x^*;y^*)$ becomes more accurate before running the algorithm (\vvs{impossible in practice}). Specifically, SI-ADMM ($\rho = 20$) is worse than DSA when $\Gamma = 50$, and better when $\Gamma \geq 5000$. 
The theoretical bound of $\bE[\|z_k - z^*\|^2]$ for SI-ADMM is also shown in  Figure~\ref{SIADMMvsDSA1}. Meanwhile, Table~\ref{tab:SIADMvsSA} displays more information of the comparison. We can see that SI-ADMM takes less time and generates better solutions given the same sample size when $\Gamma$ is relatively large. Moreover, if there is no projection onto the compact ball $Z$, DSA \vvs{produces poor results.} 


\begin{figure}[htbp!]
\begin{center}
\includegraphics[width=3.5in]{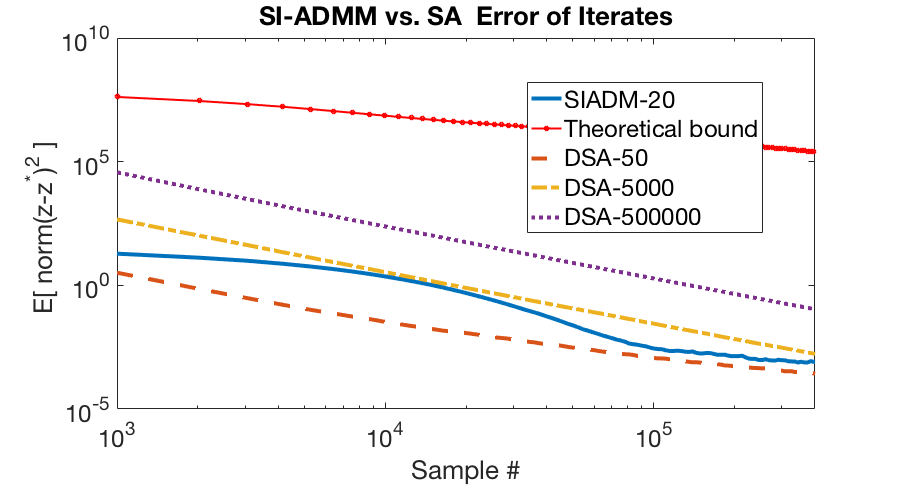}
\end{center}
\caption{SI-ADMM vs. DSA;
SIADM-20: SI-ADMM with $\rho = 20$; 
DSA-$\Gamma$: Distributed stochastic approximation with parameter $\Gamma$.
}\label{SIADMMvsDSA1}
\vspace{-0.2in}
\end{figure}

\begin{table}[htbp]
{
\begin{center}
\begin{tabular}{| c | c | c | c |}
\hline
\multicolumn{4}{|c|}{$n = 50$, sample batch $\#$ = 429139} \\
\hline
Method & Time(s) & $\bE[F(\bar z) - F^*]$ & $\bE [\| \bar z - z^* \|^2]$ \\
\hline
SIADM-20 & 46.4 & 3.08e-03 & 6.68e-04\\
\hline
DSA-5000 & 79.9 & 3.17e-03 & 1.50e-03\\
\hline
DSA-500000 & 79.7 & 1.61e-01 & 9.31e-02\\
\hline
DSA-PF & 73.5 & 2.14e+12 & 1.23e+12\\
\hline
\multicolumn{4}{|c|}{$n = 100$, sample batch $\#$ = 200028} \\
\hline
SIADM-20 & 76.8 & 1.41e-02 & 3.08e-03\\
\hline
DSA-5000 & 88.7 & 1.29e-01 & 7.23e-02\\
\hline
DSA-500000 & 88.8 & 2.27e+01 & 7.26e+00\\
\hline
DSA-PF & 88.1 & 2.69e+26 & 1.53e+26\\
\hline
\end{tabular}
\end{center}
\caption{SI-ADMM vs. DSA} 
 \label{tab:SIADMvsSA}
 The algorithms terminate when prescribed no. of samples are drawn; $\bar z$ is the updated pair $(\bar x, \bar y)$. $F(\bar z) = f(\bar x) + g(\bar y)$. Expectation approximated by sample average (10 runs). Time specified for single run. 
SIADM-20: SI-ADMM with $\rho = 20$; 
DSA-$\Gamma$: DSA with parameter $\Gamma$; 
DSA-PF: DSA without projection on a compact set.
}
\vspace{-0.25in}

\end{table}


\section{Concluding Remarks} \label{sec: conclusion-SIADMM}
\vvs{ADMM schemes have proved to be remarkably useful in solving a broad collection of structured optimization problems. Yet, much remains to be understood about how such schemes can be extended to stochastic regimes. Motivated by the gaps in schemes for contending with general structured stochastic convex problems}, we 
develop an implementable stochastic inexact ADMM (SI-ADMM) scheme based on solving the subproblems (inexactly) by stochastic approximation. We  provide almost sure convergence statements and rate guarantees for the iterates generated by this scheme under suitable assumptions on the inexactness sequence.  
Furthermore, the overall iteration complexity \vvs{is shown to be consistent
with the canonical result for stochastic approximation}.  Preliminary numerics
suggest that SI-ADMM \vs{performs} well \vs{compared} to its competitors.

\bibliographystyle{IEEEtran}
\bibliography{ref}

\section{Appendix}
\noindent{\bf Proof of Lemma~\ref{lm: main recursive inequality} \yux{(ii)}}
\begin{proof}
First let us define the following constants.
\begin{align}
C_2^x & \triangleq C_1^x \triangleq 2 \lambda_{\max}(\hat{P}) + 4\rho \gamma \lambda_{\max}(A^TA), \notag \\
\notag
\bar C_{1,2}^x & \triangleq \left( \tfrac{\gamma_x^2}{2 c_x \gamma_x -1 - \gamma_x^2M_x/K_x} + K_x b_x \right) v_1^x (1+R) \cdot \tfrac{3}{\lambda_{\rm min}(\hat{P})(1+\delta) }  
 + K_x a_x \tfrac{3}{\lambda_{\rm min}(\hat{P})} \left(1 + \left(\tfrac{1}{1+\delta}\right)\right), \notag \\
\hat C_{1,2}^x & = \hat C_{1,1}^x \triangleq  \left( \tfrac{\gamma_x^2}{2 c_x \gamma_x -1 - \gamma_x^2M_x/K_x} + K_x b_x \right) v_1^x \cdot 3(1+R)\left( \tfrac{\rho}{c_x} \| A^T B \| \right)^2  + 3 K_x a_x \left( \tfrac{\rho}{c_x} \| A^TB \| \right)^2, \notag \\
C_{1,2}^y & \triangleq \left( \tfrac{\gamma_y^2}{2 c_y \gamma_y -1 - \gamma_y^2M_y/K_y} + K_y b_y \right) v_1^y (1+R) \cdot \tfrac{2}{\lambda_{\rm min}(Q)(1+\delta) }   + K_y a_y \tfrac{2}{\lambda_{\rm min}(Q)} \left(1 + \tfrac{1}{1+\delta}\right),  \notag \\
C_2^y & \triangleq C_1^y \triangleq 2 \lambda_{\max}(\hat{P}) \left( \tfrac{\rho}{c_x} \| A^TB \| \right)^2 + \lambda_{\max}(Q) + 4 \left( \tfrac{\rho  \| A^TB \| }{c_x} \right)^2 \rho \gamma \lambda_{\max}(A^TA) 
 + 2\rho\gamma \lambda_{\max}(B^TB), \notag\\
\bar C_{1,1}^x & \triangleq \left( \tfrac{\gamma_x^2}{2 c_x \gamma_x -1 - \gamma_x^2M_x/K_x} + K_x b_x \right) ( \sv{3} v_1^x (1+R)  \| x^* \|^2 + v_2^x), \mbox{ and}\notag\\
C_{1,1}^y & \triangleq \left( \tfrac{\gamma_y^2}{2 c_y \gamma_y -1 - \gamma_y^2 M_y/K_y} + K_y b_y \right) ( \sv{2} v_1^y (1+R)  \| y^* \|^2 + v_2^y). \label{const-SIADMM}
\end{align}
Then, we show \eqref{supermartigale inequality 0} by proving that $\forall k \geq 0$, if $T_{k+1}^y \geq K_y, T_{k+1}^x \geq K_x$, then for all $R_{0,k} > 0$, 
\begin{align} \label{supermartigale inequality 0 - appendix}
\notag \bE[\| u_{k+1} - u^* \|^2_G \mid \scrF_k ] 
& \leq \left( (1+R_{0,k}) \left( \tfrac{C_2^x}{T_{k+1}^x} \left( \bar C_{1,2}^x + \hat C_{1,2}^x \tfrac{C_{1,2}^y}{T_{k+1}^y} \right) + \tfrac{C_2^y C_{1,2}^y}{T_{k+1}^y} \right) + \tfrac{1 + 1/R_{0,k}}{1 + \delta} \right) \| u_k - u^* \|_G^2 \ \\
& + (1+R_{0,k}) \left( \tfrac{C_1^x}{T_{k+1}^x } \left( \bar C_{1,1}^x + \hat C_{1,1}^x \tfrac{C_{1,1}^y}{T_{k+1}^y} \right)  + \tfrac{C_1^y C_{1,1}^y}{T_{k+1}^y} \right). 
\end{align}
Then by substitution $\zeta_2^x = C_2^x \bar C_{1,2}^x$, $\zeta_2^y = C_2^y C_{1,2}^y$, $\zeta_2^{xy} = C_2^x \hat C_{1,2}^x C_{1,2}^y$, $\zeta_1^x = C_1^x \bar C_{1,1}^x$, $\zeta_1^y = C_1^y C_{1,1}^y$, $\zeta_1^{xy} = C_1^x \hat C_{1,1}^x C_{1,1}^y$, \eqref{supermartigale inequality 0} follows.

To derive inequality \eqref{supermartigale inequality 0 - appendix},
   we need to derive  bounds on the following terms: $\bE[\| y_{k+1} - y_{k+1}^* \|^2 \mid \Fscr_k ]$,
   $\bE[\| \tilde{x}_{k+1}^* - x_{k+1}^* \|^2 \mid
   \Fscr_k]$, $\bE[ \| x_{k+1} - \tilde{x}_{k+1}^* \|^2 \mid \Fscr_k]$,
   $ \bE[ \| \lambda_{k+1} - \lambda_{k+1}^* \|^2 \mid \Fscr_k]$, and
   $\bE[ \| u_{k+1} - u_{k+1}^* \|_G^2 \mid \Fscr_k]$, \vs{where $y_{k+1}^*, x_{k+1}^*, \tilde{x}_{k+1}^*, \lambda_{k+1}^*, u_{k+1}^*$ are defined in \eqref{y_k_star-SIADMM}-\eqref{u_k_star-SIADMM}.} 

$\pmb{ \bE[\| y_{k+1} - y_{k+1}^* \|^2 \mid \Fscr_k] }$: By
	recalling that $T_{k+1}^y \geq K_y$ and from error
		bound~\eqref{y-errorbound} \yxe{in Lemma~\ref{Lm: errorbound-SIADMM}}, 
\begin{align*}
& \quad \bE[\| y_{k+1}-y_{k+1}^* \|^2 \mid \Fscr_k] 
 \leq \frac{\max \left\{\frac{\gamma_y^2 C_y^{(k+1)}}{2 c_y \gamma_y - \gamma_y^2 M_y/K_y - 1}, K_y \bar{e}_{K_y}^{y,k+1} \right\} }{T_{k+1}^y} \\ 
& \leq  \frac{\frac{\gamma_y^2 C_y^{(k+1)}}{2 c_y \gamma_y - \gamma_y^2 M_y/K_y - 1} + K_y ( a_y \| y_k - y_{k+1}^* \|^2 + b_y C_y^{(k+1)} ) }{T_{k+1}^y} \\
& =  \frac{ \left( \frac{\gamma_y^2}{2 c_y \gamma_y - \gamma_y^2 M_y/K_y - 1} + K_yb_y \right) C_y^{(k+1)} + K_y a_y \| y_k - y_{k+1}^* \|^2 }{ T_{k+1}^y} \\
& \leq \frac{ \left( \frac{\gamma_y^2}{2 c_y \gamma_y - \gamma_y^2 M_y/K_y - 1} + K_yb_y \right) [ v_1^y (1+R)\| y_{k+1}^* \|^2 + v_2^y ]  
  + K_y a_y \| y_k - y_{k+1}^* \|^2 }{T_{k+1}^y} \\
& \leq \frac{ \left( \frac{\gamma_y^2}{2 c_y \gamma_y - \gamma_y^2
		M_y/K_y - 1} + K_yb_y \right)  [ v_1^y (1+R)(2\| y_{k+1}^* - y^* \|^2 + 2 \| y^* \|^2) + v_2^y ] 
 + K_y a_y \| y_k - y_{k+1}^*
			\|^2}{T_{k+1}^y}.
\end{align*}
By invoking the definition of $C^y_{1,1}$,
			the expression on the right can be restated and bounded as follows.
\begin{align*}
& \quad \bE[\| y_{k+1}-y_{k+1}^* \|^2 \mid \Fscr_k] \\
&  \leq  \frac{ \left( \frac{\gamma_y^2}{2 c_y \gamma_y - \gamma_y^2 M_y/K_y - 1} + K_yb_y \right) v_1^y (1+R) \cdot 2 \| y_{k+1}^* - y^* \|^2  
 + K_y a_y \| y_k - y_{k+1}^* \|^2 +  C_{1,1}^y }{T_{k+1}^y} \\
&  \leq  \frac{ \left( \frac{\gamma_y^2}{2 c_y \gamma_y - \gamma_y^2 M_y/K_y - 1} + K_yb_y \right) v_1^y (1+R) \cdot 2\| y_{k+1}^* - y^* \|^2  + {2} K_y a_y (\| y_k - y^* \|^2 + \| y^* - y_{k+1}^* \|^2 ) + C_{1,1}^y}{T_{k+1}^y} \\
&  \leq \frac{ \left( \frac{\gamma_y^2}{2 c_y \gamma_y - \gamma_y^2 M_y/K_y - 1} + K_yb_y \right) \frac{2  v_1^y (1+R)}{\lambda_{\min}(Q)} \| u_{k+1}^* - u^* \|_G^2  + \frac{2 K_y a_y}{\lambda_{\min}(Q) } (\| u_k - u^* \|_G^2 + \| u^* - u_{k+1}^* \|_G^2 ) + C_{1,1}^y}{T_{k+1}^y} \\
&  \leq \frac{ \left( \frac{\gamma_y^2}{2 c_y \gamma_y - \gamma_y^2 M_y/K_y - 1} + K_yb_y \right) \frac{2  v_1^y (1+R)}{\lambda_{\min}(Q)} {\cdot} \frac{\| u_k^* - u^* \|_G^2}{1+\delta}  
 +  \frac{2 K_y a_y}{\lambda_{\min}(Q) } \left( \| u_k - u^* \|_G^2 + \frac{\| u^* - u_k^* \|_G^2}{1+\delta} \right) + C_{1,1}^y }{T_{k+1}^y} \\
& = \frac{C_{1,2}^y \| u_k - u^* \|_G^2 + C_{1,1}^y}{T_{k+1}^y},
\end{align*}
where the second inequality follows from noting that
$\|y_k-y^*_{k+1}\|^2 \leq 2\|y_k-y^*\|^2 + 2 \|y^*-y_{k+1}^*\|^2$, and
the fourth inequality is a consequence of inequality \eqref{contraction
	of iterates-SIADMM}.

$\pmb{ \bE[\| \tilde{x}_{k+1}^* - x_{k+1}^* \|^2 \mid \Fscr_k] } $:
From the definition of the updates
in $x$, in Algorithm~\ref{generalized admm} and~\ref{s-admm}, the gradient of the augmented Lagrangian is given by the following:
\begin{align}
 \label{eq-inexact update1} 
{\bf 0} & = \nabla_x \Lscr_\rho(x_{k+1}^*,y_{k+1}^*,\lambda_k) + P(x_{k+1}^* -
		 x_k)  \\
\label{eq-inexact update2}
{\bf 0} & = \nabla_x \Lscr_\rho(\tilde{x}_{k+1}^*,y_{k+1},\lambda_k) +
P(\tilde{x}_{k+1}^* - x_k)
\end{align}
Recall that $\Lscr_\rho(x,y_{k+1},\lambda_k) + \frac{1}{2}\|x - x_k \|_P^2$ is
$c_x-$strongly convex in $x$, {implying the following, where the second equality is from combining \eqref{eq-inexact update1} and \eqref{eq-inexact update2}, and fourth inequality follows from the Cauchy-Schwarz inequality.}
\begin{align}
\notag
& \quad c_x\|\tilde{x}_{k+1}^* - x_{k+1}^*\|^2 \\
\notag
& \leq  \left< \tilde{x}_{k+1}^* - x_{k+1}^*,
	\nabla_x\Lscr_\rho (\tilde{x}_{k+1}^*,y_{k+1},\lambda_k) + P(\tilde{x}_{k+1}^* - x_k) \right. \\
\notag
& \left.  -  \nabla_x\Lscr_\rho (x_{k+1}^*,y_{k+1},\lambda_k) - P (x_{k+1}^* - x_k) \right> \\
\notag
& =  \left< \tilde{x}_{k+1}^* - x_{k+1}^*,
	\nabla_x\Lscr_\rho (\tilde{x}_{k+1}^*,y_{k+1},\lambda_k) + P\tilde{x}_{k+1}^* -  \nabla_x\Lscr_\rho (x_{k+1}^*,y_{k+1},\lambda_k) - P x_{k+1}^* \right> \\
\notag
& =  \left< \tilde{x}_{k+1}^* - x_{k+1}^* , \nabla_x\Lscr_\rho(
		x_{k+1}^*,y_{k+1}^*,\lambda_k) + P(x_{k+1}^*- \tilde{x}_{k+1}^*)  \right. \\
\notag
& \left. + P \tilde{x}_{k+1}^* -
\nabla_x\Lscr_\rho(x_{k+1}^*,y_{k+1},\lambda_k) - P x_{k+1}^* \right> \\
\notag
& \leq  \left< \tilde{x}_{k+1}^* - x_{k+1}^*, \nabla_x\Lscr_\rho(
		x_{k+1}^*,y_{k+1}^*,\lambda_k) - \nabla_x\Lscr_\rho (x_{k+1}^*,y_{k+1},\lambda_k) \right> \\
\notag
& \leq  \| \tilde{x}_{k+1}^* - x_{k+1}^*\|  \| \rho A^TB(y_{k+1}^*-y_{k+1}) \| 
\notag \\
& \notag
\Rightarrow 
 c_x\|\tilde{x}_{k+1}^* - x_{k+1}^*\| \leq   \rho \|A^TB\| \|y_{k+1}-y_{k+1}^* \| \\
& \Rightarrow 
 \bE [\|\tilde{x}_{k+1}^* - x_{k+1}^*\|^2 \mid \Fscr_k] \leq
\left(\rho \|A^TB\| /c_x  \right)^2 \bE[\|y_{k+1}-y_{k+1}^* \|^2
\mid \Fscr_k].
\label{bound1-main recursive lm}
\end{align}

$\pmb{ \bE[\| x_{k+1} - \tilde{x}_{k+1}^* \|^2 \mid \Fscr_k] } $:  Based
on the assumption that $T_{k+1}^x \geq K_x$ and inequality \eqref{bound1-main recursive lm}, by taking
conditional expectations $\bE[ \bullet \mid \scrF_k]$ on both sides of
inequality \eqref{x-errorbound} \yxe{from Lemma~\ref{Lm: errorbound-SIADMM}} {and by recalling that $\scrF_k \subseteq
	\scrF_{k+1}^y$, we obtain the following sequence of inequalities.}
\begin{align*}
& \bE[ \| x_{k+1} - \tilde x_{k+1}^* \|^2 \mid \scrF_k ] 
 \leq \frac{ \frac{\gamma_x^2 \bE [C_x^{(k+1)} \mid \scrF_k ] }{2 c_x \gamma_x - \gamma_x^2 M_x/K_x - 1} + K_x \bE [ \bar{e}_{K_x}^{x,k+1} \mid \scrF_k ]  }{T_{k+1}^x } \\
& \leq \frac{ \left( \frac{\gamma_x^2}{2 c_x \gamma_x - \gamma_x^2 M_x/K_x - 1} + K_x b_x \right)[v_1^x (1 + R) \bE [ \| \tilde{x}_{k+1}^* \|^2 \mid \scrF_k] + v_2^x  ] 
+ K_x a_x \bE [ \| x_k - \tilde{x}_{k+1}^* \|^2 \mid \scrF_k ]}{T_{k+1}^x}  \\
& \leq {1\over   T_{k+1}^x}\left\{ \left( \frac{\gamma_x^2}{2 c_x \gamma_x - \gamma_x^2 M_x/K_x - 1} + K_x b_x \right) v_1^x (1 + R) \times \right. \\ 
& \bE [ 3 \| \tilde{x}_{k+1}^* - x_{k+1}^* \|^2 + 3\| x_{k+1}^* - x^*\|^2 + 3 \| x^* \|^2 \mid \scrF_k]  \\
&  + K_x a_x \bE[ 3\| x_k -x^* \|^2 + 3 \| x^* - x_{k+1}^* \|^2 + 3 \| x_{k+1}^* -\tilde{x}_{k+1}^* \|^2 \mid \Fscr_k ] \\
& + \left. \left( \frac{\gamma_x^2}{2 c_x \gamma_x - \gamma_x^2 M_x/K_x - 1} + K_x b_x \right)v_2^x \right\}, 
\end{align*}
where the third inequality follows $\|a+b+c\|^2 \leq 3\|a\|^2
		+ 3\|b\|^2 + 3\|c\|^2$. Then the right hand side (rhs) can be further bounded as follows by invoking the definition of $\bar{C}^x_{1,1}$.
\begin{align*}
& \mbox{rhs} 
 \leq {1\over   T_{k+1}^x} \bigg[ \left( \frac{\gamma_x^2}{2 c_x \gamma_x - \gamma_x^2 M_x/K_x - 1} + K_x b_x \right) {\times}  \\
&  v_1^x (1 + R) \left[ 3 \left( \frac{\rho}{c_x} \| A^T B \| \right)^2 \bE [\| y_{k+1} - y_{k+1}^* \|^2 \mid \scrF_k] + \frac{ 3  \| u_{k+1}^* - u^*\|_G^2  }{\lambda_{\min}(\hat{P}) } \right] \\
& + K_x a_x \left[ \frac{3 \| u_k - u^* \|_G^2}{\lambda_{\min}(\hat{P})} + \frac{3 \|u^* - u_{k+1}^*\|_G^2}{\lambda_{\min}(\hat{P})} + 3 \left( \frac{\rho}{c_x} \| A^TB \| \right)^2 \bE[
\|y_{k+1} - y_{k+1}^* \|^2 \mid \scrF_k ] \right]  
 + \bar{C}_{1,1}^x \bigg]. 
\end{align*}
{Next, by invoking \eqref{contraction
	of iterates-SIADMM} and by recalling the definitions \vs{in~\eqref{const-SIADMM}} and the bound for $\bE[ \| y_{k+1} - y_{k+1}^* \|^2 \mid \Fscr_k ]$, we obtain
the final bound.}
\begin{align*}
& \bE[ \| x_{k+1} - \tilde x_{k+1}^* \|^2 \mid \scrF_k ] 
  \leq {1 \over T_{k+1}^x} \left\{ \left( \frac{\gamma_x^2}{2 c_x \gamma_x - \gamma_x^2 M_x/K_x - 1} + K_x b_x \right) v_1^x (1 + R)  \times \right.\\
& \left[ 3 \left( \frac{\rho}{c_x} \| A^T B \| \right)^2 \frac{C_{1,2}^y \| u_k - u^* \|_G^2 + C_{1,1}^y}{T_{k+1}^y} + \frac{ 3  \| u_k - u^*\|_G^2 }{\lambda_{\min}(\hat{P})(1+\delta) } \right] \\
& + K_x a_x \left[ \frac{3 \| u_k - u^* \|_G^2}{\lambda_{\min}(\hat{P})} + \frac{3 \|u^* - u_k\|_G^2}{\lambda_{\min}(\hat{P})(1+\delta)} + 3 \left( \frac{\rho}{c_x} \| A^TB \| \right)^2 \frac{C_{1,2}^y \| u_k - u^* \|_G^2 + C_{1,1}^y}{T_{k+1}^y} \right]   
+ \bar{C}_{1,1}^x \bigg\} \\
& = \left[ \left[ \left( \frac{\gamma_x^2}{2 c_x \gamma_x - \gamma_x^2 M_x/K_x - 1} + K_x b_x \right) v_1^x (1 + R) + K_x a_x \right] \cdot 3 \left(\frac{\rho}{c_x}\| A^TB \| \right)^2 \frac{C_{1,2}^y}{T_{k+1}^y}  \right.\\
&  + \left( \frac{\gamma_x^2}{2 c_x \gamma_x - \gamma_x^2 M_x/K_x - 1} + K_x b_x \right) v_1^x (1 + R) \cdot \frac{3}{\lambda_{\min}(\hat{P}) (1 + \delta)} \\
& \left. + K_x a_x \frac{3}{\lambda_{\min}(\hat{P})} \left( 1 + \frac{1}{1+ \delta} \right) \right] \| u_k - u^* \|_G^2 / T_{k+1}^x \\
& + \left[ \left[ \left( \frac{\gamma_x^2}{2 c_x \gamma_x - \gamma_x^2 M_x/K_x - 1} + K_x b_x \right) v_1^x (1 + R) + K_x a_x \right] \cdot 3 \left(\frac{\rho}{c_x}\| A^TB \| \right)^2 \frac{C_{1,1}^y}{T_{k+1}^y} \right.\\
& + \left. \bar{C}_{1,1}^x \right] / T_{k+1}^x
 = \frac{\left( \hat{C}_{1,2}^x \frac{C_{1,2}^y}{T_{k+1}^y} + \bar{C}_{1,2}^x
\right) \| u_k - u^* \|_G^2 + \left( \hat{C}_{1,1}^x \frac{C_{1,1}^y}{T_{k+1}^y} + \bar{C}_{1,1}^x
\right)}{T_{k+1}^x}.
\end{align*}

$\pmb{ \bE[\| \lambda_{k+1} - \lambda_{k+1}^* \|^2 \mid \Fscr_k] } $:
Recall
that $ \lambda_{k+1}^* = \lambda_k - \gamma \rho (Ax_{k+1}^* +
		By_{k+1}^*-b),$
and this implies that
\begin{align*}
& \mathbb{E}[\|\lambda_{k+1}-\lambda_{k+1}^* \|^2 | \Fscr_k] 
 \leq 2\gamma^2 \rho^2 \left(\mathbb{E}[ \lambda_{\max}(A^TA) \|x_{k+1}^*-x_{k+1}\|^2 | \Fscr_k] + \mathbb{E}[ \lambda_{\max}(B^TB) \|y_{k+1}-y_{k+1}^*\|^2 | \Fscr_k ]\right)\\ 
& \leq 2\gamma^2 \rho^2 \left\{ \lambda_{\max}(A^TA) (2 \mathbb{E}[\|x_{k+1}^*-\tilde{x}_{k+1}^*\|^2 | \Fscr_k] \right. + 2 \bE[\| \tilde{x}_{k+1}^* - x_{k+1} \|^2 \mid \scrF_k] ) \\
&  + \left.  \lambda_{\max}(B^TB) \mathbb{E}[ \|y_{k+1}-y_{k+1}^*\|^2 | \Fscr_k ] \right\} \\
& \leq 2\gamma^2 \rho^2 \left\{  \lambda_{\max}(A^TA) \left[ 2 \left( \rho \| A^TB \| / c_x \right)^2  \mathbb{E}[\|y_{k+1} - y_{k+1}^*\|^2 | \Fscr_k]  \right. \right. \\
& + 2 \bE[\| \tilde{x}_{k+1}^* - x_{k+1} \|^2 \mid \scrF_k] \Big] +  \mathbb{E}[ \lambda_{\max}(B^TB) \|y_{k+1}-y_{k+1}^*\|^2 | \Fscr_k ] \Big\}.
\end{align*} 

$\pmb{ \bE[\| u_{k+1} - u_{k+1}^* \|_G^2 \mid \Fscr_k] } $: 
Through the definition of $G$-norm \vs{and} the upper bound developed for $\bE[ \| \lambda_{k+1} - \lambda_{k+1}^* \|^2 \mid \Fscr_k]$, we have that:
\begin{align*}
& \quad \bE[\| u_{k+1} - u_{k+1}^* \|_G^2 \mid \Fscr_k] 
 \leq \lambda_{\max}(\hat{P}) \bE[ \| x_{k+1} - x_{k+1}^* \|^2 \mid \Fscr_k] + \lambda_{\max}(Q) \bE[\| y_{k+1} - y_{k+1}^* \|^2 \mid \Fscr_k] \\
&  + \frac{1}{\rho \gamma} \bE[ \| \lambda_{k+1} - \lambda_{k+1}^* \|^2 \mid \Fscr_k] \\
& \leq \lambda_{\max}(\hat{P}) ( 2\bE[ \| x_{k+1} - \tilde{x}_{k+1}^* \|^2 \mid \Fscr_k] + 2 \bE[\| \tilde{x}_{k+1}^* - x_{k+1}^* \|^2 \mid \Fscr_k] )\\
&  + \lambda_{\max}(Q) \bE[\| y_{k+1} - y_{k+1}^* \|^2 \mid \Fscr_k] + \frac{1}{\rho \gamma} \bE[ \| \lambda_{k+1} - \lambda_{k+1}^* \|^2
\mid \Fscr_k] \\
& \leq \lambda_{\max}(\hat{P}) \left[ 2\bE[ \| x_{k+1} - \tilde{x}_{k+1}^* \|^2 \mid \Fscr_k] + 2 \left( \frac{\rho}{c_x} \| A^TB \| \right)^2 \bE[\| y_{k+1}^* - y_{k+1} \|^2 \mid \Fscr_k] \right] \\
&  + \lambda_{\max}(Q) \bE[\| y_{k+1} - y_{k+1}^* \|^2 \mid \Fscr_k] + 2 \gamma \rho \mathbb{E}[ \lambda_{\max}(B^TB) \|y_{k+1}-y_{k+1}^*\|^2 | \Fscr_k ] \\
&  +  2 \gamma \rho  \left[  \lambda_{\max}(A^TA) \left[ 2 \left(
			\rho\| A^TB \|/c_x \right)^2  \mathbb{E}[\|y_{k+1} -
		y_{k+1}^*\|^2 | \Fscr_k] \right. 
 + 2 \bE[\| \tilde{x}_{k+1}^* - x_{k+1} \|^2 \mid \scrF_k] \Big] \right], 
\end{align*}
{where the second inequality follows from $\|x_{k+1}- x_{k+1}^*\|^2
\leq 2\|x_{k+1}-\tilde x_{k+1}^*\|^2 + 2 \|x^*_{k+1}-\tilde x_{k+1}^*\|^2$, and
the third inequality \vs{follows from} \eqref{bound1-main recursive lm}. Then the bound
	on the right can be simplified as follows.}
	\begin{align*}
& \bE[\| u_{k+1} - u_{k+1}^* \|_G^2 \mid \Fscr_k] 
 \leq (2\lambda_{\max}(\hat{P}) + 4 \gamma \rho \lambda_{\max}(A^TA) ) \bE[\| \tilde{x}_{k+1}^* - x_{k+1} \|^2 \mid \Fscr_k] \\
&  + \left[ 2 \lambda_{\max}(\hat{P}) \left( \frac{\rho}{c_x} \| A^TB \| \right)^2 + \lambda_{\max}(Q) \right. \\
&  + \left. 4\gamma \rho \lambda_{\max}(A^TA) \left( \frac{\rho\| A^TB \|}{c_x} \right)^2 + 2\gamma \rho  \lambda_{\max}(B^TB) \right] \bE[ \| y_{k+1} - y_{k+1}^* \|^2 \mid \Fscr_k] \\
& = C_{1}^x \cdot \bE[\| \tilde{x}_{k+1}^* - x_{k+1} \|^2 \mid \Fscr_k] + C_1^y \cdot \bE[ \| y_{k+1} - y_{k+1}^* \|^2 \mid \Fscr_k].
\end{align*}
{From the bounds for} $\bE[ \| y_{k+1} - y_{k+1}^* \|^2$,
	$\bE[ \| \tilde{x}_{k+1}^* - x_{k+1} \|^2$ and $\bE[ \| u_{k+1} - u_{k+1}^*
	\|_G^2$, and inequality \eqref{contraction of iterates-SIADMM} \vs{and}
	\eqref{Cauchy-SIADMM}, we obtain the following :
\begin{align*}
& \bE[ \| u_{k+1} - u^* \|_G^2 \mid \scrF_k ] \\
& \leq (1 + R_{0,k}) \bE[ \| u_{k+1} - u_{k+1}^* \|_G^2 \mid \Fscr_k ] + (1 + 1/R_{0,k}) \bE[ \| u_{k+1}^* - u^* \|_G^2 \mid \Fscr_k ] \\
& \leq (1 + R_{0,k}) \bE[ \| u_{k+1} - u_{k+1}^* \|_G^2 \mid \Fscr_k ] + \frac{(1 + 1/R_{0,k})}{1 + \delta} \bE[ \| u_k - u^* \|_G^2 \mid \Fscr_k ] \\
& \leq (1 + R_{0,k})( C_{1}^x \cdot \bE[\| \tilde{x}_{k+1}^* - x_{k+1} \|^2 \mid \Fscr_k] + C_1^y \cdot \bE[ \| y_{k+1} - y_{k+1}^* \|^2 \mid \Fscr_k] ) \\
&  + \frac{(1 + 1/R_{0,k})}{1 + \delta} \| u_k - u^* \|_G^2 \\
& \leq (1 + R_{0,k})C_1^x \left[ \left( \hat{C}_{1,2} \frac{C_{1,2}^y}{T_{k+1}^y} + \bar{C}_{1,2}^x \right) \frac{ \| u_k - u^* \|_G^2}{T_{k+1}^x} + \left( \hat{C}_{1,1}^x \frac{C_{1,1}^y}{T_{k+1}^y} + \bar{C}_{1,1}^x \right) \big/ T_{k+1}^x \right] \\
&   + (1 + R_{0,k})C_1^y \cdot \frac{C_{1,2}^y \| u_k - u^* \|_G^2 + C_{1,1}^y}{T_{k+1}^y} +  \frac{(1 + 1/R_{0,k})}{1 + \delta} \| u_k - u^* \|_G^2 \\
& \leq \left[ (1+R_{0,k}) \left( \frac{C_2^x}{T_{k+1}^x} \left( \bar C_{1,2}^x + \hat C_{1,2}^x \frac{C_{1,2}^y}{T_{k+1}^y} \right) + \frac{C_2^y C_{1,2}^y}{T_{k+1}^y} \right) + \frac{1 + 1/R_{0,k}}{1 + \delta} \right] \| u_k - u^* \|_G^2 \\
&  + (1+R_{0,k}) \left( \frac{C_1^x}{T_{k+1}^x } \left( \bar C_{1,1}^x + \hat C_{1,1}^x \frac{C_{1,1}^y}{T_{k+1}^y} \right)  + \frac{C_1^y C_{1,1}^y}{T_{k+1}^y} \right).
\end{align*}
\end{proof}

\noindent{\bf Proof of Corollary~\ref{Corr: theobound}}
\begin{proof}
\yux{(i) Apply Theorem~\ref{linear convergence of g-ADMM} and the result follows.} \\
\yux{(ii)} We show this corollary by proving that for any $k \geq 0$, $R_{0,k} > 0$,
\begin{align}
 \bE[ \| u_{k+1} - u^* \|_G^2 \mid \Fscr_k] \leq \left( \frac{ (1+R_{0,k}) \zeta_2}{T_{k+1}^x} + \frac{1 + 1/R_{0,k}}{ 1 + \delta} \right) \|u_k - u^*\|_G^2 + \frac{ (1+R_{0,k}) \zeta_1}{T_{k+1}^x}, 
\end{align}
where
\begin{align} \label{Def: zeta}
\begin{aligned}
& \zeta_2 = (\lambda_{\max}(\hat{P}) + \rho\gamma \| A \|^2 ) C_{1,2}^x, \quad \zeta_1 = ( \lambda_{\max}(\hat{P}) + \rho \gamma\| A \|^2 ) C_{1,1}^x, \\
& C_{1,2}^x = \left( \frac{\gamma_x^2}{2c_x\gamma_x - 1 - \gamma_x^2 M_x / K_x } + K_x b_x \right) v_1^x (1 + R) \cdot \frac{2}{\lambda_{\min} (\hat{P})(1+\delta) } \\
& + K_x a_x \cdot \frac{2}{\lambda_{\min}(\hat{P})} \left(1+ \frac{1}{1+\delta} \right), \\
& C_{1,1}^x = \left( \frac{\gamma_x^2}{2c_x\gamma_x - 1 - \gamma_x^2 M_x / K_x } + K_x b_x \right)[ v_1^x(1+R) \cdot 2 \| x^* \|^2 + v_2^x ],
\end{aligned}
\end{align}
$R$ is any positive number that is consistent with the one in definition of $M_x$, $M_y$, $C_x^{(k)}$, $C_y^{(k)}$. (See \eqref{Def: M_y and K_y}) and \eqref{Def: M_x and K_x})

Since $y_{k+1}$ can be exactly obtained, $y_{k+1} = y_{k+1}^*$ and $x_{k+1}^* = \tilde{x}_{k+1}^*$, $\scrF_{k+1}^y = \scrF_k$. Recall that $y_{k+1}^*$, $x_{k+1}^*$, $\tilde{x}_{k+1}^*$ and $\lambda_{k+1}^*$ are defined in \eqref{y_k_star-SIADMM}, \eqref{x_k_star-SIADMM}, \eqref{x_tilde_k_star-SIADMM} and \eqref{lambda_k_star-SIADMM}. Then, 
	\begin{align*}
	& \bE[ \| x_{k+1} - x_{k+1}^* \|^2 \mid \scrF_k ]  \leq \frac{ \frac{\gamma_x^2 \bE [C_x^{(k+1)} \mid \scrF_k ] }{2 c_x \gamma_x - \gamma_x^2 M_x/K_x - 1} + K_x \bE [ \bar{e}_{K_x}^{x,k+1} \mid \scrF_k ]}{T_{k+1}^x}  \\
	& \leq  \frac{\left(\frac{\gamma_x^2}{2 c_x \gamma_x - \gamma_x^2 M_x / K_x - 1} + K_x b_x \right)[v_1^x (1 + R)  \bE [ \| x_{k+1}^* \|^2 \mid \scrF_k] + v_2^x ] 
 + K_x a_x \bE [ \| x_k - x_{k+1}^* \|^2 \mid \scrF_k ]}{T_{k+1}^x}  \\
	& = \frac{ \left( \frac{\gamma_x^2}{2 c_x \gamma_x - \gamma_x^2 M_x/K_x - 1} + K_x b_x \right)[v_1^x (1 + R) \| x_{k+1}^* \|^2 + v_2^x ]  + K_x a_x \| x_k - x_{k+1}^* \|^2 }{T_{k+1}^x}  \\
	& \leq \frac{ \left( \frac{\gamma_x^2}{2 c_x \gamma_x - \gamma_x^2 M_x/K_x - 1} + K_x b_x \right) v_1^x (1 + R) \left( 2\| x_{k+1}^* - x^*\|^2 + 2 \| x^* \|^2 \right)}{T_{k+1}^x} \\
	& + \frac{K_x a_x \left( 2\| x_k -x^* \|^2 + 2 \| x^* - x_{k+1}^* \|^2 \right)  +  \left( \frac{\gamma_x^2}{2 c_x \gamma_x - \gamma_x^2 M_x/K_x - 1} + K_x b_x \right)v_2^x}{T_{k+1}^x} \\
	& \leq \frac{ \left( \frac{\gamma_x^2}{2 c_x \gamma_x - \gamma_x^2 M_x/K_x - 1} + K_x b_x \right) v_1^x (1 + R) \cdot \left( \frac{2 \| u_k - u^*\|_G^2}{\lambda_{\min}(\hat P)(1+\delta)} + 2 \| x^* \|^2 \right)}{T_{k+1}^x} \\
	& + \frac{\frac{2 K_x a_x}{\lambda_{\min}(\hat P)}\left(1+\frac{1}{1+\delta}\right)\| u_k -u^* \|_G^2 
	 +  \left( \frac{\gamma_x^2}{2 c_x \gamma_x - \gamma_x^2 M_x/K_x - 1} + K_x b_x \right)v_2^x}{T_{k+1}^x} \\
	& = \frac{  \overbrace{\left(\left( \frac{\gamma_x^2}{2c_x\gamma_x - 1 - \gamma_x^2 M_x / K_x } + K_x b_x \right) \frac{2v_1^x (1 + R)}{\lambda_{\min}(\hat{P})(1+\delta)} + 
 \frac{2 K_x a_x}{\lambda_{\min}(\hat{P})}\left(1+
	\frac{1}{1+\delta}\right) \right)}^{\vs{ \triangleq C_{1,2}^x}} \|u_k - u^* \|_G^2}{T_{k+1}^x}  \notag\\
	& + \frac{\overbrace{\left( \frac{\gamma_x^2}{2c_x\gamma_x - 1 - \gamma_x^2 M_x / K_x } + K_x b_x \right)[ v_1^x(1+R)\cdot 2 \| x^* \|^2 + v_2^x ]}^{\vs{\triangleq C_{1,1}^x}}}{ T_{k+1}^x} 
	= \frac{ C_{1,2}^x \| u_k - u^* \|_G^2 + C_{1,1}^x }{T_{k+1}^x}.
	\end{align*}
	where the third inequality follows from $\|a+b\|^2 \leq 2\|a\|^2
		+ 2\|b\|^2$, and the fourth inequality is a result form the definition of $G$-norm as well as \eqref{contraction of iterates-SIADMM}. Furthermore, \vs{for all $R_{0,k} > 0$ and $\forall k \geq 0$}, 
	\begin{align*}
	 & \bE[ \| u_{k+1} - u^* \|_G^2 \usss{ \mid \scrF_k }] 
	 \leq (1 + R_{0,k}) \bE[ \| u_{k+1} - u_{k+1}^* \|_G^2 \mid \Fscr_k ] + (1 + \tfrac{1}{R_{0,k}}) \bE[ \| u_{k+1}^* - u^* \|_G^2 \mid \Fscr_k ] \\
	& \leq (1 + R_{0,k}) \bE[ \| u_{k+1} - u_{k+1}^* \|_G^2 \mid \Fscr_k ] + \frac{(1 + \tfrac{1}{R_{0,k}})}{1 + \delta} \bE[ \| u_k - u^* \|_G^2 \mid \Fscr_k ] \\
	& = (1 + R_{0,k}) \bE[ \| x_{k+1}^* - x_{k+1} \|_{\hat P}^2 + \frac{1}{\rho \gamma} \| \lambda_{k+1}^* - \lambda_{k+1} \|^2 \mid \Fscr_k] 
	 + \frac{(1 + \tfrac{1}{R_{0,k}})}{1 + \delta} \| u_k - u^* \|_G^2 \\
	& = (1 + R_{0,k}) \bE[ \| x_{k+1}^* - x_{k+1} \|_{\hat P}^2 + \frac{1}{\rho \gamma} \| \rho \gamma A(x_{k+1}^* - x_{k+1}) \|^2 \mid \Fscr_k] 
	+ \frac{(1 + \tfrac{1}{R_{0,k}})}{1 + \delta} \| u_k - u^* \|_G^2 \\
	& \leq (1 + R_{0,k}) (\lambda_{\max}(\hat P) + \rho \gamma \|A\|^2 ) \bE[ \| x_{k+1}^* - x_{k+1} \|^2 \mid \Fscr_k] 
	 + \frac{(1 + \tfrac{1}{R_{0,k}})}{1 + \delta} \| u_k - u^* \|_G^2 \\
	& \leq (1 + R_{0,k}) (\lambda_{\max}(\hat P) + \rho \gamma \|A\|^2 ) \cdot \frac{C_{1,2}^x \|u_k - u\|_G^2 + C_{1,1}^x }{T_{k+1}^x} 
	 + \frac{(1 + \tfrac{1}{R_{0,k}})}{1 + \delta} \| u_k - u^* \|_G^2 \\
	& = \Bigg[ \frac{(1 + R_{0,k}) \overbrace{(\lambda_{\max}(\hat P) + \rho \gamma \|A\|^2 ) C_{1,2}^x}^{\vs{ \triangleq \zeta_2}}}{T_{k+1}^x} + \frac{(1 + \tfrac{1}{R_{0,k}})}{1 + \delta} \Bigg] \cdot \| u_k - u^* \|_G^2 \\
	&  + \frac{(1 + R_{0,k}) \overbrace{(\lambda_{\max}(\hat P) + \rho \gamma \|A\|^2 )C_{1,1}^x}^{\vs{\triangleq \zeta_1}}}{T_{k+1}^x}, 
	\end{align*}
where the second, \vs{third, and fourth inequalities are based on \eqref{contraction of iterates-SIADMM},  the property of matrix norm, and  the bound} for $\bE[ \| x_{k+1} - x_{k+1}^* \|^2 \mid
\scrF_k ]$ developed earlier.  \end{proof} 

\end{document}